\documentclass[a4paper,reqno,11pt]{amsart}
\usepackage[english]{babel}
\usepackage{amssymb,amsmath,amsfonts,amsthm,mathabx,enumitem,dsfont}
\usepackage{graphicx}
\usepackage[usenames,dvipsnames]{xcolor}

\usepackage[colorlinks=true,linkcolor=blue,citecolor=blue,filecolor=blue,urlcolor=blue,pageanchor=true,plainpages=false,linktocpage]{hyperref}

\numberwithin{equation}{section}
\newtheorem{proposition}{Proposition}[section]
\newtheorem{theorem}[proposition]{Theorem}

\newtheorem{lemma}[proposition]{Lemma}

\theoremstyle{definition}
\newtheorem{definition}[proposition]{Definition}
\newtheorem{remark}[proposition]{Remark}

\newcommand{\N}{\mathbb{N}}

\newcommand{\R}{\mathbb{R}}

\newcommand{\eps}{\varepsilon}

\makeatletter
\newcommand{\mylabel}[2]{#2\def\@currentlabel{#2}\label{#1}}
\makeatother
\newcommand{\hi}{$(1)$}
\newcommand{\hii}{$(2)$}
\newcommand{\hiii}{$(3)$}
\newcommand{\hiv}{$(4)$}
\newcommand{\hv}{$(5)$}
\newcommand{\hvi}{$(6)$}
\newcommand{\ci}{$(i)$}
\newcommand{\cii}{$(ii)$}
\newcommand{\ciii}{$(iii)$}

\newcommand{\ind}[1]{\mathds{1}_{#1}}

\title[On a degenerate second order traffic model]{On a degenerate second order traffic model: \\ existence of discrete evolutions, deterministic many-particle limit and first order approximation}

\author[D. Mazzoleni, E. Radici and F. Riva]{Dario Mazzoleni, Emanuela Radici and Filippo Riva}

\begin{document}
	
	\begin{abstract}
		We propose and analyse a new microscopic second order Follow-the-Leader type scheme to describe traffic flows. The main novelty of this model consists in multiplying the second order term by a nonlinear function of the global density, with the intent of considering the attentiveness of the drivers in dependence on the amount of congestion. 
		Such term makes the system highly degenerate; indeed, coherently with the modellistic viewpoint, we allow for the nonlinearity to vanish as soon as consecutive vehicles are very close to each other. We first show existence of solutions to the degenerate discrete system. We then perform a rigorous discrete-to-continuum limit, as the number of vehicles grows larger and larger, by making use of suitable piece-wise constant approximations of the relevant macroscopic variables.
		The resulting continuum system turns out to be described by a degenerate pressure-less Euler-type equation, and we discuss how this could be considered an alternative to the groundbreaking Aw-Rascle-Zhang traffic model.
		Finally, we study the singular limit to first order dynamics in the spirit of a vanishing-inertia argument. This eventually validates the use of first order macroscopic models with nonlinear mobility to describe a congested traffic stream. 
	\end{abstract}
	
	\maketitle
	
	\noindent
	{\footnotesize \textbf{2020 AMS-Subject Classification: 
		}76A30, 34A12, 35Q70, 90B20, 34E15}.
	
	\tableofcontents
	
	\section{Introduction}
	The study of mathematical models for traffic flows has considerably spread in the recent decades due to the twofold need of preventing vehicular congestion on road networks and, at the same time, managing the increasing demand for mobility. 
	Moreover, these models contain several challenging and interesting mathematical issues to tackle.
	For a general introduction to the literature on this topic we refer to the books \cite{GaravelloPiccoliBook,HabermannBook,RosiniBook}
	and to the survey papers \cite{Helbing,PiccoliTosin}. We now present a non exhaustive overview of those articles related to the features more closely linked to this work.
	
	\subsection{Overview of the literature and of the existing models}
	
	Mathematical models for traffic can be classified according to several features which become relevant or not depending on the envisaged application.  Among such features, we mention the level of description of the system (microscopic, mesoscopic, macroscopic) and the representation of the process (deterministic, stochastic). 
	Deterministic microscopic models give a detailed description of the traffic flow and characterise individually each vehicle and its interaction with the other vehicles and the environment. However, they are computationally too costly when the number of vehicles is big and do not provide insights into the relevant features of the traffic flow. Deterministic macroscopic models, on the other hand, describe traffic as a flow without distinguishing the constituent vehicles. 
	They are computationally less demanding as they involve a limited number of variables (density, velocity and flow). Moreover, they are very suited for analysing macroscopic characteristics of vehicular traffic, like shock waves and queue lengths.
	
	Macroscopic models further split into the two main subcategories of \emph{equilibrium} and \emph{non-equilibrium} models, depending on the relation between the velocity and the density of vehicles. 
	In equilibrium models (or \emph{first-order} models) the velocity is expressed as a function of the density, while in non-equilibrium ones (or \emph{second-order} ones) density and velocity are coupled through a partial differential equation. 
	First-order models should be preferred when describing applications where the shock structures are irrelevant, because they do not take into account the distribution of the desired velocities across vehicles and, therefore, are not able to predict many traffic instabilities (e.g. stop \& go waves, hysteresis phenomena, phantom jams, etc). Among first-order models we recall the Lighthill-Whitham-Richards (LWR) model \cite{LighthillWhitham,Richards}, which represents the starting point for the modelling of vehicular flows, and its extension to road networks as introduced in~\cite{HilligesWeidlich,HoldenRisebro} and further studied in \cite{BrianiCristiani}.
	The second-order model which is currently the most used is called Aw-Rascle-Zhang (ARZ) and has been introduced independently by Aw and Rascle \cite{AwRascle} and Zhang \cite{Zhang}. It has been proposed to solve the inconsistencies (pointed out in \cite{Daganzo}) of a previous second-order model by Payne \cite{Payne} and Whitham \cite{Whitham}.
	We mention that macroscopic models for traffic are often treated in the fluid dynamics framework, as the traffic stream can be regarded as a fluid flow, see for instance~\cite{Zatorskaetal1,Zatorskaetal2}.
	
	The present work is more concerned with the link between microscopic and macroscopic models.
	Such connection for the LWR model has been well established in a series of works \cite{CardaliaguetForcadel,ColomboRossi,DiFrancescoRosini,GoatinRossi,HoldenRisebro2, HoldenRisebro3,Rossi} in the framework of many-particle limits of first-order Follow-the-Leader models. 
	Second-order Follow-the-Leader type models, instead, have been investigated by 
	Aw et al. \cite{AwKlarMaterneRascle}, Greenberg \cite{Greenberg} and Di Francesco et al. \cite{DiFrancescoFagioliRosiniTRAFFIC} in connection with the ARZ model. However, these latter results all deal with the simplified homogeneous version of the ARZ model (see Section~\ref{sec:ARZ} for more details), thus leaving the many-particle derivation of the complete nonlinear model still an open question. We also mention that modifications of the homogeneous ARZ model have been considered in the literature, for example the \emph{generalised} Aw-Rascle-Zhang model (GARZ) introduced in \cite{FanHertySeibold}. A second-order scheme has been proposed in \cite{ChiarelloFriedrichGoatinGoettlich} to study the micro-macro limit for GARZ.

	We conclude this overview by mentioning related works on the extensions of the aforementioned models to road networks and to multi-scale analysis. 
	Among multi-scale models, which couples the traffic description at different scales separated by an interface that can be either fixed or solution-dependent, we mention \cite{LattanzioPiccoli} for a second-order micro and macro coupling and \cite{CristianiIacomini} for second-order micro and first-order macro one. 
	On the other hand, many-particle limits for road networks traffic flows were considered in \cite{CristianiSahu} in the case of first order Follow-the-Leader schemes.
	Consistently, as observed in \cite{HilligesWeidlich} and further investigated in \cite{BrettiBrianiCristiani,BrianiCristiani}, the corresponding macroscopic models appear to be suitable extensions of the LWR model to networks.

	\subsection{A new second order model}
	In the present work we propose a novel second-order deterministic individual-based model to describe traffic dynamics with congestion. 
	First, we prove an existence result for a new microscopic scheme, then we study the convergence, as the number of individuals grows larger and larger, towards a second-order macroscopic model resulting in a hydrodynamic pressure-less Euler-type system. Finally, we investigate the passage to the first-order dynamics in the spirit of a vanishing-inertia asymptotics. This aims at validating the use of first-order schemes, reducing the level of complexity of the problem still being close to the original system.

	Our modellistic starting point lies in the observation that the drivers' reaction-time to inputs should be directly proportional to the local amount of congestion. 
	We postulate that the drivers' attention is higher in overcrowded regions while, in areas of lower density (and thus also of lower danger) the level of alertness may decrease. 
	
	Precisely, we consider a population of $N+1$ (thinking) indistinguishable individuals (vehicles), located at positions $x_0(t), \ldots, x_N(t)$ along the evolution, moving in one fixed direction according to the following system of second-order ordinary differential equations
	\begin{equation}\label{eq:intro schema}
		\eps \zeta(\rho_i(t))\ddot{x}_i(t) + \gamma\dot{x}_i(t) = \vartheta(\rho_i(t))F(t,x_i(t)) ,\qquad t\in(0,T),
	\end{equation}
	which we will show to fall in the class of Follow-the-Leader type models (that is, when the traffic dynamics is governed by the interaction between a vehicle and the vehicle immediately in front).
	
	The right-hand side of~\eqref{eq:intro schema} represents the forcing term, and is the product of two quantities: a congestion function $\vartheta$ and a drift $F$. The simplest, though most popular, choice for drifts in traffic flow models is $F\equiv 1$, modelling a free road with speed limit normalised to 1. The (possibly nonlinear) function $\vartheta$ describes the amount of congestion and is evaluated at $\rho_i(t)$, representing a local reconstruction of the global particle density at $x_i(t)$. 
	We consider approximations $\rho_i$ of the form 
	\begin{equation}\label{eq:rhoi}
		\rho_i(t) = \frac{1}{N(x_{i+1}(t) - x_i(t))},
	\end{equation}
	namely proportional to the inverse of the distance between consecutive vehicles. Indeed, according to the one-directional motion of vehicular flows on a line, the traffic density detected by the $i$-th particle depends on its proximity to the particle ahead. 
	Hard congested traffic models are characterised by a threshold density above which the particles are stuck and cannot move anymore. This feature is encoded in~\eqref{eq:intro schema} by requiring $\vartheta$ 
	to be nonincreasing and compactly supported on $[0,\bar{\rho}_\vartheta]$; due to \eqref{eq:rhoi}, one may interpret the constant $\bar{\rho}_\vartheta$, here called hard congestion threshold, as the renormalised inverse of the vehicles' lenght (see also \eqref{eq:lenght}). We emphasize that the monotonicity of $\vartheta$ attenuates the effect of the drift $F$ as the density increases. 
	Typical examples of $\vartheta$ appearing in traffic flows \cite{LighthillWhitham}, pedestrian flows \cite{DiFrancescoMarkowichPietschmannWolfram}, animal swarming \cite{AggarwalGoatin,ColomboMercierGaravello}, bacterial chemotaxis \cite{DiFrancescoRosado}, are given by $\vartheta(\rho) = (1 - \rho)^{\alpha}_+$ for $\alpha \geq 1$.
	
	For what concerns the left-hand side of~\eqref{eq:intro schema}, $\gamma,\eps > 0$ are fixed parameters, $\dot{x}_i, \ddot{x}_i$ denote the velocity and the acceleration of the $i$-th particle respectively, while $\zeta$ is another nonnegative function of the density $\rho_i$ describing the effect of the congestion on the drivers' response-time, namely on their alertness. 
	Indeed, in the literature, the presence of second order terms in individual-based systems usually models the attentiveness of the particle to the surrounding environment. 
	Smaller values of the coefficient in front of $\ddot{x}_i$ correspond to a faster response of the particle; in the limit case where the coefficient is null, i.e. there is no second order term in the equation, the velocity of the particle instantaneously changes accordingly to the forcing term. 
	For this reason, when modelling traffic flows, it feels natural to require that the coefficient of the second order term depends on how crowded the space is in front of the driver in the direction of the motion. The term $\eps \zeta(\rho_i) \ddot{x}_i$  in~\eqref{eq:intro schema} has precisely this role: in crowded regions (where $\rho_i$ is close to or above a certain threshold $\bar{\rho}_\zeta$) we expect $\zeta$ to be very small or even null, thus the particle will almost instantaneously adapt its velocity, while in low-density areas ($\rho_i$ close to $0$) we expect $\zeta \approx 1$ as the level of attention drops and the particle can slowly react to inputs. 
	Accordingly to these observations, we require $\zeta$ to be supported on an interval $[0,\bar{\rho}_\zeta]$ for some constant $\bar{\rho}_\zeta \leq \bar{\rho}_\vartheta$ somehow describing an instantaneous response threshold: whenever $\rho_i(t)\ge \bar{\rho}_\zeta$ the equation \eqref{eq:intro schema} becomes of first order, and so the velocity $\dot x_i(t)$ changes instantaneously according to the right hand-side.
	We emphasize that the alertness function $\zeta$ plays a similar role to that of $\vartheta$ in accounting for the local level of congestion, however the two functions may differ from each other since the effect of the congestion on the forcing term or on the second order term may, in principle, not be related at all. 
	
	\subsection{Main results}
	In this work we first establish the well-posedness of  the system~\eqref{eq:intro schema} under general assumptions on $\zeta, \vartheta, F$ and the initial data. We may also consider different functions $\zeta_i, \vartheta_i, F_i$  for each particle, modelling distinct individual responses (see Remark~\ref{rmk:individual}). Notice that the presence of $\zeta$ in front of the leading order term $\ddot{x}_i$ makes the equation highly degenerate since $\zeta$ can vanish when $x_i$ is too close to the particle ahead.
	Such kind of degeneracy leads to several mathematical difficulties and prevents one from applying standard arguments to prove existence. Indeed, although theories on singular ordinary differential equations are known \cite{BookDegODE}, we are not aware of existing settings which include the system under consideration. 
	Moreover, equation~\eqref{eq:intro schema} makes sense only if $\rho_i$ can be defined, namely when the distance $x_{i+1}- x_i$ remains positive along the evolution, see~\eqref{eq:rhoi}.
	This is an expected behaviour of the system, indeed when $x_i$ and $ x_{i+1}$ are very close, $\zeta$ and $\vartheta$ both vanish and formally $\dot{x}_i = 0$. Surprisingly, the rigorous proof of this observation is far from being trivial and requires a careful study of the equations.  
	Another crucial feature of microscopic traffic flow models is that all particles move in the same direction. We show that the proposed scheme fulfills the expected behaviour, despite this being not evident from the expression~\eqref{eq:intro schema}. 
	
	We then study uniform bounds for suitable approximations of some relevant macroscopic quantities, e.g. density, first and second order moment, whose definitions depend on the particles trajectories solving~\eqref{eq:intro schema}. 
	The presence of the nonlinearities $\zeta,\vartheta$ prevents from approximating the macroscopic density $\rho$ in the continuum setting via the standard approach via empirical measures $\frac{1}{N+1} \sum_{i=0}^{N} \delta_{x_i}$ as done in \cite{CarrilloChoiHauray} for first-order dynamics and in \cite{CarrilloChoi} for second-order ones. 
	
	Exploiting the one-dimensional structure of the dynamics, we instead rely on piece-wise constant approximation techniques as those introduced in~\cite{DiFrancescoRosini} and further extended to more involved first-order transport mechanisms, including nonlocal interactions~\cite{DiFrancescoFagioliRadici,FagioliTse,RadiciStra}, external potentials~\cite{DiFrancescoStivaletta} and diffusive operators~\cite{DaneriRadiciRuna2,FagioliRadici}, and to deal with the second-order homogeneous ARZ model~\cite{DiFrancescoFagioliRosiniTRAFFIC}. We show that such approximations of the density and the first and second moment enjoy suitable uniform bounds, and thus also weak compactness in suitable topologies.
	
	We then perform the many-particle limit of the system~\eqref{eq:intro schema} as the number $N$ of vehicles goes to infinity in the regime $\zeta = k \vartheta$, for some $k>0$. In this case, we derive the following continuum degenerate pressureless Euler-type model 
	\begin{equation}\label{eq:intro continuum}
		\begin{cases}
			\partial_t\rho+\partial_x(\rho u)=0,&\text{in }(0,T)\times \R, \\
			\eps k \vartheta(\rho)\Big(\partial_t(\rho u)+\partial_x(\rho u^2)\Big) + \gamma\rho u=\rho\vartheta(\rho)F,&\text{in }(0,T)\times \R,
		\end{cases}
	\end{equation}
	where $\rho$ represents the traffic density and $u$ the average velocity. In this setting, we are able to prove that the weak limits obtained by the discrete approximations of the density $\rho$ and first and second moment $\rho u$ and $\rho u^2$ respectively, solve system~\eqref{eq:intro continuum}, at least in a very weak sense (see Definition~\ref{def:solcont}). 
	
	To the best of our knowledge, this is the first convergence result for second-order models involving the nonlinearity $\vartheta$.
	We do not consider here the many-particle limit for $\zeta \neq \vartheta$, since this more general case requires stronger compactness properties of the piece-wise interpolants which are highly nontrivial and will be treated in future works. 
	However, the uniform bounds satisfied by the discrete approximations do not depend explicitly on the parameter $\eps$ and this allows for the degeneracy of the small parameter $\eps = \eps_N$ in the many-particle limit.
	We can thus perform a combined many-particle limit and asymptotic analysis as $\eps\to0$ to prove that the piece-wise interpolations converge in the limit as $N \to \infty$, and hence $\eps_N \to 0$, to a suitable weak solution of the first order macroscopic traffic model
	\[ 
	\partial_t \rho +\frac 1\gamma\partial_x(\rho\vartheta(\rho) F)=0,\qquad\text{in }(0,T)\times \R,
	\]
	in the sense of Theorem~\ref{thm:vaninert}.
	
	\subsection{Comparison with the ARZ model}\label{sec:ARZ}
	We conclude this introduction with a comparison between the proposed model and the ARZ model where, for simplicity, we set $F \equiv 1$ and the parameters $\eps =\gamma = 1$. In this case, the ARZ model in its macroscopic nonlinear formulation reads as 
	\begin{equation}\label{eq: ARZ model macro}
		\begin{cases}
			\partial_t\rho+\partial_x(\rho u)=0,&\text{in }(0,T)\times \R, \\
			\partial_t(\rho w)+\partial_x(\rho u w) = A \rho \big( \vartheta(\rho) - u\big),&\text{in }(0,T)\times \R,
		\end{cases}
	\end{equation}
	where $A\ge 0$ is a nonnegative constant, while $w$ is usually referred to in the literature as a \emph{Lagrangian marker} expressing a sort of average desired velocity of the vehicles flow and is related to the actual average velocity through the relation $w = u + P(\rho)$, where $P$ is a pressure function depending on the density.
	
	In the ARZ model, the congestion effects play a role exactly in the latter relation. Indeed, by the monotonicity of $P$, higher values of $\rho$ (i.e. crowded regions) correspond to  higher values of the pressure, and so to higher mismatch between the actual and the desired velocities $u$ and $w$.
	Our model, on the other hand, directly deals with the first and second moment related to the actual average velocity $u$ and encompasses congestion effects via the alertness function $\zeta$. Indeed, in the same setting $F \equiv 1$ and $\eps=\gamma=1$, the system \eqref{eq:intro continuum} (here we are not requiring $\zeta = k \vartheta$) can be written in the following form 
	\begin{equation*}
		\begin{cases}
			\partial_t\rho+\partial_x(\rho u)=0,&\text{in }(0,T)\times \R, \\
			\zeta(\rho)\Big(\partial_t(\rho u)+\partial_x(\rho u^2)\Big) =\rho \big( \vartheta(\rho) - u \big),&\text{in }(0,T)\times \R.
		\end{cases}
	\end{equation*}
	
	We recall that many-particle limit results for the ARZ model \eqref{eq: ARZ model macro} are nowadays available only in the homogeneous case $A = 0$ and, therefore, deal with microscopic models of the form 
	\begin{equation}\label{eq: ARZ model micro hom}
		\begin{cases}
			\dot{x}_i(t) = w_i(t) - P(\rho_i(t)), \\
			\dot{w}_i(t) = A \big( \vartheta(\rho_i(t)) - \dot{x}_i(t) \big) = 0,
		\end{cases}\qquad t\in(0,T),
	\end{equation}
	which is actually a first order system since from the second equation $w_i$ turns out to be constant along the evolution. Thus, it is possible to apply many first-order tools \cite{DiFrancescoFagioliRosiniTRAFFIC} to deal with the homogeneous problem. 
	To the best of our knowledge, in the present work we consider for the first time existence and many-particle limit results for a second-order traffic model involving the nonlinearity $\vartheta$. Indeed, system \eqref{eq:intro schema} is truly of second-order nature as it can also be seen by the following reformulation (still when $F \equiv 1$ and $\eps=\gamma=1$)
	\begin{equation}\label{eq:simplifiedmicro}
		\begin{cases}
			\dot{x}_i(t) = w_i(t) , \\
			\zeta (\rho_i(t)) \dot{w}_i(t) = \vartheta(\rho_i(t)) - \dot{x}_i(t),
		\end{cases}\qquad t\in(0,T).
	\end{equation}
	As a common feature between the microscopic models \eqref{eq: ARZ model micro hom} (also in the non-homogeneous case $A>0$) and~\eqref{eq:simplifiedmicro}, we mention that in both cases the term $\vartheta(\rho_i)$ may represent an \emph{equilibrium velocity} for the $i$-th particle.   
	
	\smallskip
	
	{\bf Plan of the paper.} In Section~\ref{sec:setting} we detail the degenerate microscopic traffic model we propose and analyse in the paper. We list all the required assumptions and we state our main results regarding existence of solutions, the limit passage to a continuum second order traffic model as the number of particles increases, and the asymptotic analysis as $\eps\to0$ leading to a first order model with nonlinear mobility. We refer respectively to Theorems~\ref{thm:maindiscrete},~\ref{thm:maincont} and~\ref{thm:vaninert}. In Section~\ref{sec:trafficlight} we present an explicit example describing the situation of a traffic light that becomes red. Section~\ref{sec:lemmas} collects useful results concerning the behaviour of solutions to degenerate second order differential equations which will be employed in the work. The last three sections are devoted to the proofs of the main results: we prove the existence theorem for the discrete system in Section~\ref{sec:existdiscrete}, the many-particle limit in Section~\ref{sec:manyparticle}, and finally the first-order approximation in Section~\ref{sec:vaninertia}.
	
	\smallskip 
	
	{\bf Notations.}	The maximum (resp. minimum) of two extended real numbers $\alpha,\beta\in \R\cup\{\pm\infty\}$ is denoted by $\alpha\vee\beta$ (resp. $\alpha\wedge\beta$). The interior of a set $A$ is denoted by $\mathring A$ or ${\mathrm{int}}(A)$.
	
	For functions $f(t)$ depending on one scalar variable representing time, we denote by $\dot f(t)$ its derivative. If $f=f(t,x)$ instead, we write $\partial_t f$, $\partial_x f$ for the partial derivatives with respect to time $t$ and space $x$, respectively.
	
	We adopt standard notations for Lebesgue and Sobolev spaces or spaces of continuous, Lipschitz continuous, or continuously differentiable functions. A superscript $^+$ is added when referring to subspaces of nonnegative functions. By $\mathcal M([0,T]\times \R)$ we denote the set of Radon measures on $[0,T]\times \R$, and we write $\mathcal M([0,T]\times \R)^+$ for its subset of positive measures.

	\section{Setting and main results}\label{sec:setting}
	\subsection{Degenerate microscopic traffic model}\label{subsec:discrete}
	The microscopic traffic flow scheme we consider in this paper models the evolution of $N+1$ ordered vehicles (or thinking particles), whose position at time $t\in [0,T]$, for some time horizon $T>0$, is represented by the function $x_i(t)$, for $i=0,\ldots, N$. In order to describe the system, we introduce several quantities. The distance between two consecutive vehicles is denoted by $d_i(t):=x_{i+1}(t)-x_i(t)$, for $i=0,\ldots,N-1$, and we set
	\begin{equation}\label{eq:rho}
		\rho_i(t):=\frac{1}{N d_i(t)}=\frac{1}{N (x_{i+1}(t)-x_i(t))},\qquad \text{for }i=0,\ldots N-1,
	\end{equation}
	modelling a reconstruction of the macroscopic particle density at $x_i(t)$.
	
	We also consider a nonnegative time-dependent external drift $F\in C^0([0,T]\times \R)$, while the alertness function $\zeta\in C^{0}([0,\infty))$ and the congestion function $\vartheta\in C^{0}([0,\infty))$ are assumed to be nonnegative and to satisfy 
	\[
	\zeta(r)=0\quad\text{if and only if}\quad r\geq \bar \rho_\zeta,\qquad \text{and}\qquad\vartheta(r)=0\quad\text{if and only if}\quad r\geq \bar \rho_\vartheta,
	\] 
	for certain values $\bar \rho_\vartheta\ge \bar\rho_\zeta>0$ modelling the hard congestion threshold and the instantaneous response threshold, respectively. If $\rho_i(t)\in[\bar\rho_\zeta,\bar\rho_\vartheta]$, we say that the particle $x_i(t)$ is saturated at time $t$, otherwise we say it is unsaturated. In the particular case $\rho_i(t)=\bar\rho_\vartheta$ we instead say that the particle is congested.
	
	Given an initial configuration $x_0^0<x_1^0<\ldots<x_N^0$ such that $\tfrac{1}{N(x_{i+1}^0-x_i^0)}=:\rho_i^0\leq \bar \rho_\vartheta$ for $i=0,\ldots,N-1$, we define the set of saturated indexes at the initial time as \[
	\Sigma^0=\Big\{i\in\{0,\ldots,N-1\} : \rho_i^0\in[\bar \rho_\zeta, \bar \rho_\vartheta]\Big\}.
	\]
	
	We allow the particle in front to move freely, namely we assume that $x_N\in C^0([0,T])$ is a given nondecreasing function which fulfils $x_N(0)=x_N^0$. 
	
	The traffic scheme we want to study consists in the following degenerate second order system of nonlinear ordinary differential equations, set in the time interval $[0,T]$:
	\begin{equation}\label{eq:system}
		\begin{cases}
			\eps\zeta(\rho_i(t))\ddot{x}_i(t) + \gamma\dot{x}_i(t) = \vartheta(\rho_i(t))F(t,x_i(t)),&\qquad \text{for }i=0,\ldots,N-1,\\
			x_i(0)=x_i^0,&\qquad \text{for }i=0,\ldots,N-1,\\
			\dot{x}_i(0)=v_i^0,&\qquad \text{for }i\notin \Sigma^0,
		\end{cases}
	\end{equation}
	where $\rho_i(t)$ has been introduced in \eqref{eq:rho}, $\eps,\gamma>0$ are two positive parameters and $v_i^0\geq 0$, for $i\notin \Sigma^0$, are nonnegative initial velocities. Note that prescribing initial velocities is meaningful only for unsaturated initial indexes, since for them the system is really of second order. If instead $\rho_i^0\in[\bar \rho_\zeta, \bar \rho_\vartheta]$, i.e. $i\in \Sigma^0$, then $\zeta$ vanishes at $t=0$ and the system collapses to a first order one, so there is no need to enforce an initial velocity (which actually, as we will see, may make the system ill-posed).
	
		Up to our knowledge, system \eqref{eq:system} does not fit within known classes of degenerate ordinary differential equations. Unlike the simplest and most studied kind of singularities, just depending on the independent variable $t$ \cite{Dias, Kozhanov}, our system degenerates with respect to the unknown itself (via the density $\rho_i$ as in \eqref{eq:rho}). Although even this latter type of singularities have been partially analysed, in known frameworks \cite{BookDegODE} the degeneracy occurs just on a discrete set of positions, while here the function $\zeta$ vanishes on a whole interval.
	
	The rigorous definition of solution to~\eqref{eq:system} that we adopt in this paper is the following one.
	
	\begin{definition}\label{def:sol}
		Under the previous assumptions, we say that $x=(x_0,\ldots,x_{N-1})\in C^{0,1}([0,T];\R^{N})$ is a solution to the microscopic traffic model if each component $x_i$, for $i=0,\ldots,N-1$, solves \eqref{eq:system} in the sense that:
		\begin{enumerate}
			\item[\mylabel{hi}{\hi}] $x_i(0)=x_i^0$;
			\item[\mylabel{hii}{\hii}] the map $t\mapsto x_i(t)$ is nondecreasing in $[0,T]$;
			\item[\mylabel{hiii}{\hiii}]  for all $t\in[0,T]$ there hold $x_i(t)< x_{i+1}(t)$ and $\rho_i(t)\leq \bar \rho_\vartheta$;
			\item[\mylabel{hiv}{\hiv}] the function $x_i$ is of class $C^2$ in the (relatively) open set
			\begin{equation}\label{eq:open1}
				\{t\in [0,T] : \rho_i(t)<\bar\rho_\zeta\},
			\end{equation}
			and therein it solves the second order equation
			\begin{equation}\label{eq:systempos}
				\eps\zeta(\rho_i(t))\ddot{x}_i(t)+\gamma\dot{x}_i(t)=\vartheta(\rho_i(t))F(t,x_i(t));
			\end{equation}
			\item[\mylabel{hv}{\hv}] the function $x_i$ is of class $C^1$ in the (relatively) open set
			\begin{equation}\label{eq:open2}
				\{t\in [0,T] : \rho_i(t)\in(\bar \rho_\zeta,\bar \rho_\vartheta]\}\cup{\mathrm{int}}{\{t\in [0,T] : \rho_i(t)={{\overline{\rho}}_\zeta}\}},
			\end{equation}
			and therein it solves the first order equation
			\begin{equation}\label{eq:systempos1ord}
				\gamma\dot{x}_i(t)=\vartheta(\rho_i(t))F(t,x_i(t));
			\end{equation}
			\item[\mylabel{hvi}{\hvi}] $\dot{x}_i(0)=v_i^0$ for $i\notin \Sigma^0$, namely for the unsaturated indexes, which satisfy $\rho_i^0<\bar \rho_\zeta$.
		\end{enumerate}
	\end{definition}
	\begin{remark}
		Conditions \ref{hiv} and \ref{hv}, together with the attainment of the initial data \ref{hi} and \ref{hvi}, completely specify the behaviour of the solution in the whole interval $[0,T]$. Indeed, in the remaining set $\partial \{t\in [0,T] : \rho_i(t)={{\overline{\rho}}_\zeta}\}$, where a differential equation can not be imposed since the set does not contain any interval, by definition \eqref{eq:rho} of $\rho_i$ one has
		\begin{equation*}
			x_i(t)=x_{i+1}(t)-\frac{1}{N\bar\rho_\zeta}.
		\end{equation*}
		Hence, the position of $x_i$ is prescribed by the position of $x_{i+1}$.
	\end{remark}
	
		\begin{remark}
			Condition~\ref{hiii} entails that the distance between two consecutive vehicles is always greater than or equal to a certain positive constant (depending on $N$), which can thus be interpreted as the normalised vehicles' length. Namely,
			\begin{equation}\label{eq:lenght}
				d_i(t)=\frac{1}{N \rho_i(t)}\geq \frac{1}{N\bar \rho_\vartheta}.
			\end{equation}
		\end{remark}	
	
	We state now the first result of the paper, ensuring existence of solutions to the microscopic traffic scheme. We will prove it in Section~\ref{sec:existdiscrete}. Not surprisingly, since the idea is natural when dealing with degenerate problems \cite{BookDegODE}, our approach relies on a regularization of \eqref{eq:system} and on a careful convergence analysis of the solutions to this approximated problem.
	\begin{theorem}\label{thm:maindiscrete}
		Under the previous assumptions, there exists a solution $x$ to the microscopic traffic model~\eqref{eq:system} in the sense of Definition~\ref{def:sol} satisfying the following bound on the velocity for $i=0,\ldots,N-1$:
		\begin{equation}\label{eq:uniflipbound}
			\dot x_i(t)\le \widetilde v_i^0\vee \left(\frac{\bar\vartheta}{\gamma}\,\max\limits_{[0,T]\times[x_0^0,x_N(T)]}F\right),\qquad\text{for almost every }t\in (0,T),
		\end{equation}
		where we set
		\begin{equation}\label{eq:vtilde}
			\widetilde v_i^0:=\begin{cases}
				v_i^0,\qquad &\text{if }i\not\in \Sigma^0,\\
				0,\qquad &\text{if }i\in \Sigma^0,
			\end{cases}
		\end{equation}
		and $\bar\vartheta:=\max\limits_{[0,\bar\rho_\vartheta]}\vartheta$.
		
		If in addition $x_N$ is strictly increasing and $F$ is positive, then $x_i$ is strictly increasing for all $i=0,\ldots,N-1$.
	\end{theorem}
	
	\begin{remark}\label{rmk:individual}
		Since the proof of the above theorem is based on an induction argument on the particles, the result still remains true even if the alertness and the congestion functions may be different vehicle by vehicle. Namely we could consider $\zeta_i$ and $\vartheta_i$ for $i=0,\ldots,N-1$, thus modelling distinct individual responses. For the same reason, also the drift may depend on the particle, i.e. we could choose functions $F_i$ for $i=0,\ldots,N-1$, even though from the modellistic viewpoint a unique function is more realistic since it should describe external effects.
	\end{remark}
	
	By slightly strenghtening the assumptions, we also show that every solution of the microscopic traffic model never reaches the hard congestion threshold $\bar\rho_\vartheta$. This property will be crucial for the many-particle limit and the first-order approximation we will perform later.
	
	\begin{proposition}\label{prop:nosaturated}
		In addition to the previous assumptions, suppose that the external drift $F$ is positive, that the alertness function $\zeta$ is Lipschitz continuous near $\bar \rho_\zeta$, and that $x_N\in C^1([0,T])$ fulfils $\dot x_N(t)>0$ for all $t\in (0,T]$.
		
		Then, any solution $x$ in the sense of Definition~\ref{def:sol} which furtherly fulfils 
		\begin{equation}\label{eq:isolatedpoints}
			\text{if $\bar \rho_\zeta<\bar \rho_\vartheta$, then the set $\partial \{t\in [0,T] : \rho_i(t)={{\overline{\rho}}_\zeta}\}$ is finite for all $i=0,\ldots,N-1$},
		\end{equation}
		satisfies:
		\begin{itemize}
			\item $x$ belongs to $C^1([0,T];\R^N)$;
			\item $\rho_i(t)<\bar \rho_\vartheta$ and $\dot{x}_i(t)>0$ for all $t\in (0,T]$ and $i=0,\ldots,N-1$;
			\item $\dot{x}_i(0)=\frac 1\gamma\vartheta(\rho_i^0)F(0,x_i^0)$ for $i\in \Sigma^0$.
		\end{itemize}
		Moreover, if $\bar \rho_\zeta=\bar \rho_\vartheta$ one also has $x\in C^2((0,T];\R^N)$.
	\end{proposition}
	\begin{remark}
		The request \eqref{eq:isolatedpoints} is a conditional assumption we directly put on the solution $x$. It forbids wild oscillations of $\rho_i$ around the value $\bar\rho_\zeta$ along the evolution. Although we strongly believe that every solution of the microscopic traffic model should satisfy~\eqref{eq:isolatedpoints} (at least far from $t=0$), we are not able to show its validity. However, in the case $\bar \rho_\zeta=\bar \rho_\vartheta$, which will be the case on which we focus in the sequel when dealing with the many-particle limit and the first-order approximation, that conditional assumption is not needed and the result is completely applicable.
	\end{remark}
	
	\begin{remark}\label{rmk:nosaturated}
		In case $\bar \rho_\zeta=\bar \rho_\vartheta$, the above proposition states that a solution $x$ of the microscopic traffic model is composed by particles moving with positive velocity ($\dot x_i(t)>0$) and which are never congested ($\rho_i(t)<\bar\rho_\vartheta$). In particular, they solve equation \eqref{eq:systempos} for all times $t\in (0,T)$.
	\end{remark}
	
	\begin{remark}[\textbf{Partial uniqueness}]
		In case $\bar \rho_\zeta=\bar \rho_\vartheta=:\bar \rho$, assuming in addition that $\zeta$ and $\vartheta$ are Lipschitz continuous in $[0,\bar\rho]$ and that $x\mapsto F(t,x)$ is Lipschitz continuous in $\R$ uniformly with respect to $t\in [0,T]$, if the initial particles are not congested, namely $\rho_i^0<\bar \rho$ for all $i=0,\ldots, N-1$, then the solution $x$ to the microscopic traffic model~\eqref{eq:system} is unique and of class $C^2([0,T];\R^N)$. This follows by an application of the classical Cauchy-Lipschitz Theorem, since in this situation all the particles satisfy $\rho_i(t)\le c<\bar\rho$, and thus system \eqref{eq:systempos} is not degenerate anymore.
	\end{remark}
	
	\subsection{Discrete-to-continuum limit}
	The first problem we aim to analyse concerns the asymptotic behaviour of the microscopic system \eqref{eq:system} as the number of individuals $N$ grows bigger and bigger. It is well-known \cite{AwKlarMaterneRascle,CarrilloChoi}, and natural, that such limiting behaviour can be captured by means of a system of partial differential equations of Euler-type. 
	This latter continuum point of view is usually called macroscopic approach to traffic models \cite{GaravelloPiccoliBook,RosiniBook}.
	
	In our setting, the many-particle limit of \eqref{eq:system} formally reads as the following system composed by a transport equation coupled with a degenerate partial differential equation:
	\begin{equation}\label{eq:contlimit}
		\begin{cases}
			\partial_t{\rho}+\partial_x(\rho u)=0,&\text{in }(0,T)\times \R,\\
			\eps\zeta(\rho)\Big(\partial_t{(\rho u)}+\partial_x(\rho u^2)\Big)+\gamma\rho u=\rho\vartheta(\rho)F,&\text{in }(0,T)\times \R,\\
			\rho(0)=\rho^0,\qquad (\rho u)(0)=e_1^0.
		\end{cases}
	\end{equation}
	Above, $\rho=\rho(t,x)$ represents the density of vehicles, while $u=u(t,x)$ is their average velocity. The product terms $\rho u$ and $\rho u^2$ instead are called first and second moments, respectively. Note that the initial conditions $(\rho^0,e_1^0)$ are given in terms of the density and of the first moment.
	
	We restrict our attention to the situation 
	\begin{equation}\label{eq:zeta=theta}
		\zeta\equiv\vartheta,
	\end{equation}
	so that in particular there holds
	\begin{equation}\label{eq:barrho}
		\bar\rho_\zeta=\bar\rho_\vartheta=:\bar\rho.
	\end{equation}
	
	Since the congestion function $\vartheta$ may vanish, we are again led to introduce a weak notion of solution for the macroscopic traffic model. The definition below is motivated by the fact that, whenever $\vartheta(\rho)>0$, the second equation in \eqref{eq:contlimit} can be written as \[
	\eps  \Big(\partial_t{(\rho u)}+\partial_x(\rho u^2)\Big)+\gamma\frac{\rho u}{\vartheta(\rho)}=\rho F.
	\]
	
	\begin{definition}\label{def:solcont}
		In case \eqref{eq:zeta=theta}, given an initial pair $(\rho^0,e_1^0)\in (L^\infty(\R)^+)^2$ with compact support, we say that a quadruple $(\rho, e_1,e_2,\lambda)$ is a \emph{measure solution} to the macroscopic traffic model~\eqref{eq:contlimit} if:
		\begin{enumerate}
			\item[\mylabel{ci}{\ci}] $\rho, e_1,e_2\in L^\infty((0,T)\times\R)^+$ and $\lambda\in \mathcal M([0,T]\times \R)^+$;
			\item[\mylabel{cii}{\cii}] $\|\rho\|_{L^\infty((0,T)\times \R)}\leq \bar \rho$ and the (essential) supports of $\rho,e_1,e_2,\lambda$ are compact;
			\item[\mylabel{ciii}{\ciii}]for all $\varphi\in C^1([0,T]\times \R)$ with $\varphi(T)\equiv 0$ there holds 
			\begin{equation}\label{eq:weakcontlimit}
				\begin{cases}
					\displaystyle\int_0^T\int_{\R}(\rho\partial_t{\varphi}+e_1\partial_x\varphi)\,dxdt=-\displaystyle\int_{\R}\rho^0\varphi(0)\,dx,\\
					\eps\displaystyle\int_0^T\int_{\R}(e_1\partial_t{\varphi}+e_2\partial_x\varphi)\,dxdt-\gamma\int_{[0,T]\times\R}\varphi\,d\lambda=-\eps\displaystyle\int_{\R}e_1^0\varphi(0)\,dx-\int_0^T\int_{\R}\rho F \varphi\,dxdt.
				\end{cases}
			\end{equation}
		\end{enumerate}
	\end{definition}
	
	The presence of the nonlinearity $\vartheta$ in the model prevents us from performing the many-particle limit via standard approximations by empirical measures~\cite{CarrilloChoi, CarrilloChoiHauray}. We instead adopt a technique introduced in \cite{DiFrancescoRosini} (see also \cite{DaneriRadiciRuna2,RadiciStra}), which exploits piece-wise constant approximations. In view of Definition~\ref{def:solcont}, the quantities that we aim to approximate and pass to the limit are the density and the first and second moments. For technical reasons, it is convenient to choose a mixed piece-wise constant and piece-wise linear approximation for the second moment, instead of a piece-wise constant one. Let us also stress (only in this and in the next section) the dependence on $N$ of the involved discrete quantities, by adding a supplementary exponent $(N)$.
	
	Let $x^{(N)}$ be a solution of the microscopic traffic model provided by Theorem~\ref{thm:maindiscrete}, under the additional assumptions of Proposition~\ref{prop:nosaturated}. In particular, by \eqref{eq:barrho}, the solution $x^{(N)}\in C^1([0,T];\R^N)\cap C^2((0,T];\R^N)$ satisfies $\rho_i^{(N)}(t)<\bar \rho$ for all $i=0,\ldots,N-1$. For $(t,x)\in [0,T]\times \R$, we then set
	\begin{equation}\label{eq:definterpolants}
		\begin{split}
			\rho^N(t,x)&:=\sum_{i=0}^{N-1}\rho_i^{(N)}(t)\ind{[x_i^{(N)}(t),x^{(N)}_{i+1}(t))}(x),\\
			e^N_1(t,x)&:=\sum_{i=0}^{N-1}\rho_i^{(N)}(t)\dot{x}^{(N)}_i(t)\ind{[x_i^{(N)}(t),x^{(N)}_{i+1}(t))}(x),\\
			e^N_2(t,x)&:=\sum_{i=0}^{N-1}\rho_i^{(N)}(t)\dot{x}^{(N)}_i(t)\left(\dot{x}^{(N)}_i(t)+\frac{\dot{x}^{(N)}_{i+1}(t)-\dot{x}^{(N)}_{i}(t)}{x_{i+1}^{(N)}(t)-x_i^{(N)}(t)}(x-x_i^{(N)}(t))\right)\ind{[x_i^{(N)}(t),x_{i+1}^{(N)}(t))}(x).
		\end{split}
	\end{equation}
	
	In addition to the assumptions of section~\ref{subsec:discrete} and Proposition~\ref{prop:nosaturated}, we also require that there are universal constants $s,S,C>0$ such that 
	\begin{equation}\label{eq:initialassumption}
		s\leq x_0^{0(N)}\leq x_N^{(N)}(T)\leq S,\qquad \sup_{i\notin \Sigma^{0(N)}}v_i^{0(N)}\leq C,\qquad \sup_{t\in [0,T]}\dot{x}_N^{(N)}(t)\leq C.
	\end{equation}
	
	We are now in position to state the second result of the paper, regarding the convergence, as $N$ diverges to infinity, of the microscopic system \eqref{eq:system} to the macroscopic one \eqref{eq:contlimit} when $\zeta=\vartheta$. The proof of the theorem below will be given in Section~\ref{sec:manyparticle}.
	
	\begin{theorem}\label{thm:maincont}
		In addition to the hypotheses of Theorem~\ref{thm:maindiscrete} and Proposition~\ref{prop:nosaturated}, let us assume \eqref{eq:zeta=theta}, \eqref{eq:initialassumption}, and that the map $x\mapsto F(t,x)$ is Lipschitz continuous in $\R$ uniformly with respect to $t\in [0,T]$. Then, the following convergences as $N\to\infty$ hold, up to a nonrelabelled subsequence, for the quantities introduced in \eqref{eq:definterpolants}:
		\begin{subequations}
			\begin{align}
				&x_0^{(N)}\rightarrow \underline x,&& x_N^{(N)}\rightarrow \overline x,\qquad\text{uniformly in }[0,T];\label{eq:limitsa}\\
				&\rho^{N}(0)\overset{*}{\rightharpoonup} \rho^0,&& e_1^N(0)\overset{*}{\rightharpoonup} e_1^0,\quad\,\text{weakly$^*$ in $L^\infty(\R)$};\label{eq:limitsb}\\
				& \rho^N\overset{*}{\rightharpoonup} \rho,&& e_k^N\overset{*}{\rightharpoonup}e_k,\qquad\,\,\text{weakly$^*$ in $L^\infty((0,T)\times\R)$, for $k=1,2$};\label{eq:limitsc}\\
				&\frac{e_1^N}{\vartheta(\rho^N)}\,dtdx\overset{*}{\rightharpoonup} \lambda,&& \qquad\qquad\qquad\,\,\,\text{weakly$^*$ in $\mathcal M([0,T]\times\R)$}.\label{eq:limitsd}
			\end{align} 
		\end{subequations}
		
		Moreover, $\rho^0$ and $e_1^0$ are nonnegative and have compact support, $\|\rho^0\|_{L^\infty(\R)} \leq \bar{\rho}$ and the limit quadruple $(\rho,e_1,e_2,\lambda)$ is a measure solution to the macroscopic traffic model~\eqref{eq:contlimit} with initial data $(\rho^0,e_1^0)$ in the sense of Definition~\ref{def:solcont}.
		
		Finally, the (essential) supports of the four limit objects are contained in the set\[
		\bigcup_{t\in[0,T]}\{t\}\times [\underline x(t),\overline x(t)].
		\]
	\end{theorem}

	\begin{remark}
		We point out that in \eqref{eq:limitsd} dividing by $\vartheta(\rho^N)$ is allowed since Proposition~\ref{prop:nosaturated} and Remark~\ref{rmk:nosaturated} ensure that $\vartheta(\rho_i(t))>0$ for all $t\in (0,T]$ and $i=0,\ldots, N-1$.
	\end{remark}
	
		\begin{remark}
			We stress that Theorem~\ref{thm:maincont} mainly states the weak$^*$ convergence of the quantities $\rho^N,e_k^N,$ etc., to certain limits. The fact that such limits solve \eqref{eq:contlimit} in the very weak sense of Definition~\ref{def:solcont} is somehow a byproduct of the convergence result. Indeed, requiring that four unrelated objects satisfy two partial differential equations \eqref{eq:weakcontlimit} is a rather poor requirement.
			
			We finally observe that, as the Reader may easily verify, our strategy also yields weak$^*$ convergence in $L^\infty((0,T)\times\R)$ of the discrete velocities
			\begin{equation*}
				u^N(t,x):=\sum_{i=0}^{N-1}\dot{x}^{(N)}_i(t)\ind{[x_i^{(N)}(t),x^{(N)}_{i+1}(t))}(x),
			\end{equation*}
			to a limit velocity $u$. Since $e_1^N=\rho^N u^N$ and somehow $e_2^N=\rho^N (u^N)^2$, it is natural to conjecture that the structure $e_k=\rho u^k$ shall be preserved also in the limit. This is certainly true if the convergence of $u^N$ to $u$ is strong, but this property is still under study (we actually expect also $\rho^N$ to converge strongly).
			
			A second natural conjecture, which we are not able to verify for the moment, regards the measure $\lambda$: we expect that its density with respect to the Lebesgue measure is exactly $\frac{e_1}{\vartheta(\rho)}\chi_{\{\rho<\bar\rho\}}$.
	\end{remark}
	
	\subsection{Asymptotics to first-order dynamics}
	A further problem we want to discuss is the joint many-particle limit and vanishing-inertia type analysis for the discrete model~\eqref{eq:system}. More precisely, we now allow the small parameter $\varepsilon$ in front of the second-order derivative to depend on $N$, i.e. $\eps=\eps_N$, and we let $\eps_N \to 0$ as $N \to \infty$. In this way we aim at recovering the first order traffic model with mobility analysed in \cite{DiFrancescoRosini,DiFrancescoStivaletta,FagioliTse,RadiciStra}, thus rigorously justifying it as an approximation of a more precise second order model. Indeed, at least formally, by letting the parameter $\eps \to 0$ in~\eqref{eq:contlimit}, one obtains the following system
	\begin{equation}\label{eq:vaniert}
		\begin{cases}
			\partial_t{\rho}+\partial_x(\rho u)=0,&\text{in }(0,T)\times \R,\\
			\gamma \rho u=\rho\vartheta(\rho)F,&\text{in }(0,T)\times \R,\\
			\rho(0)=\rho^0,
		\end{cases}
	\end{equation}
	which can be equivalently written as 
	\begin{equation*}
		\begin{cases}
			\partial_t{\rho}+\frac{1}{\gamma}\partial_x(\rho\vartheta(\rho) F)=0,&\text{in }(0,T)\times \R,\\
			\rho(0)=\rho^0,
		\end{cases}
	\end{equation*}
	which is exactly the problem considered in \cite{DiFrancescoRosini} when $F=1$ and $\gamma=1$.
	
	We will show that, still in the particular case \eqref{eq:zeta=theta}, the joint limit $\eps_N \to 0$, $N \to \infty$ of the microscopic system~\eqref{eq:system} provides a weak solution of~\eqref{eq:vaniert}, in a sense specified below. The notion of solution we recover in the limit is motivated by the formal observation that, whenever $\vartheta(\rho) > 0$, the second equation in~\eqref{eq:vaniert} reads as \[
	\frac 1\gamma\rho F = \frac{e_1}{\vartheta(\rho)},
	\]
	where $e_1=\rho u$ represents the first moment.
	
	More precisely, in Section~\ref{sec:vaninertia} we will prove the following result.
	\begin{theorem}\label{thm:vaninert}
		Consider the same assumptions as in Theorem~\ref{thm:maincont} and let $\eps= \eps_N$ in the particle system~\eqref{eq:system} be such that $\eps_N \to 0$ as $N \to \infty$. Then, up to a nonrelabelled subsequence, the following convergences as $N \to \infty$ hold for the quantities introduced in \eqref{eq:definterpolants}: 
		\begin{subequations}
			\begin{align}\label{eq:limits4}
				&x_0^{(N)}\rightarrow \underline x,\qquad\quad\qquad x_N^{(N)}\rightarrow \overline x,&&\text{uniformly in }[0,T];\\
				&\rho^{N}(0)\overset{*}{\rightharpoonup} \rho^0,&&\text{weakly$^*$ in $L^\infty(\R)$};\label{eq:limits4b}\\
				& \rho^N\overset{*}{\rightharpoonup} \rho, \qquad\qquad\qquad e_1^N\overset{*}{\rightharpoonup}e_1, &&\text{weakly$^*$ in $L^\infty((0,T)\times\R)$}; \label{eq:limits4c}   \\
				&\frac{e_1^N}{\vartheta(\rho^N)}\,dtdx\overset{*}{\rightharpoonup} \frac 1\gamma \rho F\,dtdx, &&\text{weakly$^*$ in $\mathcal M([0,T]\times\R)$}.\label{eq:limits4d}
			\end{align}
		\end{subequations}
		
		Moreover, $\|\rho^0\|_{L^\infty(\R)} \leq \bar{\rho}$ and $\rho^0$ is nonnegative and has compact support; also, $\rho$ and $e_1$ are nonnegative, $\|\rho\|_{L^\infty((0,T)\times \R)}\leq \bar \rho$, and their supports are contained in the set 
		\[
		\bigcup_{t\in[0,T]}\{t\}\times [\underline x(t),\overline x(t)]. 
		\]
		
		Finally, the pair $(\rho,e_1)$ solves the transport equation with initial datum $\rho^0$ in the sense that for all $\varphi \in C^1([0,T] \times \R)$ with $\varphi(T) \equiv 0$ one has 
		\begin{equation}\label{eq:CE}
			\int_0^T \int_\R (\rho \partial_t{\varphi} +e_1\partial_x\varphi)\,dxdt = -\int_{\R}\rho^0\varphi(0)\,dx.
		\end{equation}
	\end{theorem}
	
		\section{A toy example of a traffic light}\label{sec:trafficlight}
		In this section we show how our second order traffic model may describe in a realistic way the situation of a traffic light which switches to red and after a certain time switches back to green. For the sake of simplicity, we consider the case of three vehicles, so $N=2$, but we point out that this example can be easily extended to an arbitrary number of vehicles.
		
		We fix $\eps=\gamma=1$, and we define the alertness and congestion functions $\zeta,\vartheta$ as
		\begin{equation}\label{eq:xitheta}
			\zeta(\rho)=\begin{cases}
				1,\qquad &\text{if }\rho\in[0,\underline\rho_\zeta],\\
				\frac{\overline \rho_\zeta-\rho}{\overline \rho_\zeta-\underline\rho_\zeta}, &\text{if }\rho\in[\underline\rho_\zeta,\overline \rho_\zeta],\\
				0,\qquad &\text{if }\rho\in[\overline\rho_\zeta,+\infty),
			\end{cases}
			\qquad 
			\vartheta(\rho)=\begin{cases}
				1,\qquad &\text{if }\rho\in[0,\underline\rho_\vartheta],\\
				\frac{\overline \rho_\vartheta-\rho}{\overline \rho_\vartheta-\underline\rho_\vartheta}, &\text{if }\rho\in[\underline\rho_\vartheta,\overline \rho_\vartheta],\\
				0,\qquad &\text{if }\rho\in[\overline\rho_\vartheta,+\infty),
			\end{cases}
		\end{equation}
		where $0<\underline\rho_\zeta<\overline \rho_\zeta<+\infty$, $0<\underline\rho_\vartheta<\overline \rho_\vartheta<+\infty$ are fixed constants satisfying $\bar\rho_\vartheta\ge \bar \rho_\zeta$.
		
		Let us also set $V>0$ as the speed limit of the road. We now propose a function $F_{TL}\colon \R\to [0,V]$ which intends to model a red traffic light placed at the origin $x=0$. It is defined as
		\begin{equation}\label{eq:FTL}
			F_{TL}(x)=
			\begin{cases}
				V,\qquad&\text{if }x\in(-\infty,-S_2),\\
				-\frac{V}{S_2-S_1}(x+S_1),&\text{if }x\in[-S_2,-S_1),\\
				0,\qquad &\text{if }x\in[-S_1,0),\\
				\frac{V}{\delta}x,&\text{if }x\in[0,\delta),\\
				V,\qquad &\text{if }x\in[\delta,+\infty),
			\end{cases}
		\end{equation}
		where $\delta<V$ is a very small parameter, while $S_2>S_1>2(V+2\delta)$.
		
		The idea is that if a vehicle is very far from the red traffic light (on its left), namely $x<-S_2$, then it is not affected by it and so the forcing effect is the same as in a free road $F_{TL}\equiv V$. Instead, if the vehicle comes into a certain range from the red traffic light, namely $x\in [-S_1,0)$, the forcing term equal to $0$ makes it brake. We note that $S_1$ depends on $V$, which is reasonable since the space needed to brake before a traffic light depends clearly on the speed of the vehicle and thus in particular on the speed limit of the road. Finally, if the vehicle is to the right of the traffic light, then again it experiences a free road $F_{TL}\equiv V$. Observe that in \eqref{eq:FTL} there are continuous transitions between the three regimes described above, but we stress that these are not essential and their only role is to make $F_{TL}$ continuous (and clearly they can be chosen smooth, if needed).
		
		The switching between green and red mechanism is introduced through the following time-dependent function:
		\begin{equation}\label{Ftrafficlight}
			F(t,x)=\begin{cases}
				V,\qquad &\text{if }t\in [0,t_{gr}),\\
				V+\frac{F_{TL}(x)-V}{\widetilde t_{gr}-t_{gr}}(t-t_{gr}), &\text{if }t\in [t_{gr},\widetilde t_{gr}),\\
				F_{TL}(x),&\text{if }t\in [\widetilde t_{gr},t_{rg}),\\
				F_{TL}(x)+\frac{V-F_{TL}(x)}{\widetilde t_{rg}-t_{rg}}(t-t_{rg}), &\text{if }t\in [t_{rg},\widetilde t_{rg}),\\
				V,\qquad &\text{if }t\in [\widetilde t_{rg},T].
			\end{cases}
		\end{equation}
		The time $t_{gr}$ denotes the moment in which the traffic light from green becomes red. Thus, before $t_{gr}$, the vehicles move freely on the road ($F\equiv V$), while after the traffic light switches to red (and up to a linear transition) they enter in the setting governed by $F_{TL}$ described above. As we will see below, if a vehicle is in the range $[-S_1,0)$, it will remain there until the traffic light turns back to green at time $t_{rg}$. Up to another short transition, the vehicles start again to experience a free road, and thus they accelerate toward the maximal speed $V$.
		
		Let us now show how the three vehicles behave following system \eqref{eq:system} with drift $F$ as in \eqref{Ftrafficlight}. For the sake of clarity, we set
		\[
		t_{gr}=1+\frac{2\delta}{V},\qquad \widetilde t_{gr}=1+\frac{4\delta}{V},\qquad t_{rg},T>\!\!>1,
		\]
		and we choose the initial position of the vehicles as 
		\[
		x_2^0=-\delta, \qquad x_1^0=-2(V+2\delta),\qquad x_0^0<x_1^0-\frac{1}{3\overline \rho_\zeta}.
		\]
		We also assume that $V\geq \frac{1}{6\overline \rho_\vartheta}$, so that $\rho_i^0\le\overline \rho_\vartheta$ for $i=0,1$. The initial velocities can be chosen arbitrarily within the speed limit, namely
		\[
		v_i^0\in [0,V],\qquad \text{for all }i=0,1,2.
		\]
		Thanks to Theorem~\ref{thm:maindiscrete}, there exists a solution to the traffic model~\eqref{eq:system} with this drift and these initial data. Moreover, the choice of $F$ allows us to deduce, again from Theorem~\ref{thm:maindiscrete}, that $0\leq \dot{x}_i(t)\leq V$ for all $i=0,1,2$, thus all vehicles respect the speed limit.
		
		\subsection{The vehicle in front}
		In the first time interval $[0,t_{gr}]$ the vehicle $x_2$ solves (note that $\zeta(\rho_2)\equiv 1$, since $x_2$ is the first vehicle)
		\begin{equation}\label{eq:free}
			\ddot{x}_2(t)+\dot{x}_2(t)=V,
		\end{equation}
		which has the explicit solution 
		\begin{equation}\label{eq:freesol}
			x_2(t)=x_2^0-(V-v_2^0)(1-e^{-t})+Vt.
		\end{equation}
		Moreover, the time $t_{gr}$ has been chosen so that \[
		x_2(t_{gr})\geq x_2^0+V(t_{gr}-1)=\delta>0,
		\]
		namely $x_2$ has overcome the traffic light before it switched to red, and thus equation \eqref{eq:free} is actually solved in the whole $(0,T)$, whence \eqref{eq:freesol} describes the evolution of $x_2$ for all times $t\in [0,T]$.
		
		\subsection{The vehicle in between}
		This vehicle will be influenced by the traffic light becoming red. First of all, due to \eqref{eq:uniflipbound}, we note that for $t\in[0,t_{gr})$ there holds
		\[
		x_1^0\leq x_1(t)\leq x_1^0+Vt_{gr}=-2(V+2\delta)+V+2\delta=-(V+2\delta)=\frac{x_1^0}{2}.
		\]
		Analogously, in the transition region $[t_{gr},\widetilde t_{gr}]$, we have 
		\[
		x_1^0\leq x_1(t)\leq \frac{x_1^0}{2}+V(\widetilde t_{gr}-t_{gr})=\frac{x_1^0}{2}+2\delta= -V,
		\] 
		and in particular $x_1$ is in the region where $F_{TL}=0$ when the traffic light becomes red (and after the transition regime). Until the traffic light is red, the vehicle $x_1$ thus solves the equation\[
		\zeta(\rho_1(t))\ddot{x}_1(t)+\dot{x}_1(t)=0,
		\]
		whence, following formula \eqref{eq:doty}, its velocity satisfies 
		\[
		0\leq \dot{x}_1(t)=\dot{x}_1(\widetilde t_{gr})e^{-\int_{\widetilde t_{gr}}^t\frac{1}{\zeta(\rho_1(\tau))}\,d\tau}\leq Ve^{-(t-\widetilde t_{gr})}.
		\]
		As a consequence, we infer 
		\[
		x_1(t)=x_1(\widetilde t_{gr})+\int_{\widetilde t_{gr}}^t\dot{x}_1(\tau)\,d\tau\leq x_1(\widetilde t_{gr})+V(1-e^{-(t-\widetilde t_{gr})})\le 0,\qquad \text{for } t\in [\widetilde t_{gr},t_{rg}].
		\]
		
		Summing up, when the traffic light switches to red $x_1$ starts braking, its speed tends exponentially to zero, and the vehicle asymptotically reaches a certain position, always remaining to the left of the traffic light. After $t_{rg}$, namely when the traffic light turns green, since $t_{rg}$ and $T$ are very large and $\widetilde t_{rg}$ is very close to $t_{rg}$ we surely have
		\[
		\zeta(\rho_1(t))=\vartheta(\rho_1(t))=1,\qquad\text{in }[\widetilde t_{rg},T],
		\]
		thus $x_1$ solves in $[\widetilde t_{rg},T]$ the following equation\[
		\ddot{x}_1(t)+\dot{x}_1(t)=V,
		\]
		hence therein it has the explicit form\[
		x_1(t)=x_1(\widetilde t_{rg})-(V-\dot{x}(\widetilde t_{rg}))(1-e^{-(t-\widetilde t_{rg})})+V(t-\widetilde t_{rg}).
		\]
		This means that $x_1$ starts accelerating towards velocity $V$ as soon as the traffic light switches back to green.
		
		\subsection{The vehicle behind}
		The last vehicle will be affected both by the traffic light and by its proximity with the second vehicle $x_1$. The evolution of $x_0$ in $[0,\widetilde t_{gr})$ follows its equation and it is not explicit, but let us assume, in order to make the situation interesting, that $x_0(\widetilde t_{gr})\geq -S_1$, so that it has reached a position where the effect of the red traffic light can be felt. Hence, in $(\widetilde t_{gr},t_{rg})$ it solves the equation
		\[
		\zeta(\rho_0(t))\ddot{x}_0(t)+\dot{x}_0(t)=0.
		\]
		In the very same way as for $x_1$, also the vehicle $x_0$ is thus braking and its speed exponentially decreases to zero under the effect of the red traffic light. In addition, since the bound $\rho_0(t)\le \bar\rho_\vartheta$ must be satisfied by the solution, it follows that $x_0(t)\le x_1(t)-\frac{1}{3\bar\rho_\vartheta}$, namely $x_0$ remains spaced to the left of $x_1$.
		Moreover, if for some time $\overline t<t_{rg}$ it happens that $\rho_0(\overline t)=\overline \rho_\zeta$, namely if $x_0$ becomes too close to $x_1$, then from the equation one has $\dot{x}_0(\overline t)=0$ and the vehicle stops immediately. In this case, we then deduce that $\dot{x}_0(t)=0$ for all $t\in [\bar t,t_{rg}]$. This means that, if the vehicle $x_0$ stops because it was too close to the vehicle in front of it, then it restarts only when the traffic light becomes green again (and the right-hand side of the equation becomes positive again).

	\section{Preliminary tools}\label{sec:lemmas}
	
	In this section we collect some elementary but useful results regarding the behaviour of solutions to the second order differential equation
	\begin{equation}\label{eq:e1}
		\beta(t)\ddot{y}(t)+\dot y(t)=\alpha(t)F(t,y(t)),\qquad \text{for }t\in (a,b),
	\end{equation}
	where $F\in C^0([a,b]\times\R)$, $\alpha,\beta\in C^0([a,b])$ and $\beta$ is positive in $(a,b)$. The link between \eqref{eq:e1} and \eqref{eq:systempos} is evident.
	
	The first lemma provides a representation formula for the derivative $\dot y$ which will be used often throughout the paper.
	
	\begin{lemma}\label{le:lemmadoty}
		Let $y\in C^2(a,b)$ be a solution to \eqref{eq:e1}. Then, for all $a<s\leq t<b$, the following representation formula holds true:
		\begin{equation}\label{eq:doty}
			\dot y(t)=\dot y(s) e^{-\int_s^t\frac{1}{\beta(r)}\, dr}+\int_s^t \frac{\alpha(\tau)F(\tau,y(\tau))}{\beta(\tau)}e^{-\int_\tau^t\frac{1}{\beta(r)}\, dr}\, d\tau.
		\end{equation}
		In particular, if in addition $\alpha$ and $F$ are nonnegative, and if $y\in C^1([a,b))$ fulfils $\dot y(a)\ge 0$, then for all $t\in [a,b)$ one has
		\begin{equation}\label{eq:bounddoty}
			0\le \dot y(t)\le \dot y(a)\vee \left(\max\limits_{\tau\in[a,t]} \alpha(\tau) F(\tau,y(\tau))\right).
		\end{equation}
	\end{lemma}
	
	\begin{proof}
		In order to obtain \eqref{eq:doty} it is enough to divide equation~\eqref{eq:e1} by $\beta(\tau)$ and to multiply it by $e^{\int_s^\tau\frac{1}{\beta(r)}\, d r}$, deducing 
		\begin{equation*}
			\frac{{\rm d}}{{\rm d}\tau}\left(\dot y(\tau)e^{\int_s^\tau\frac{1}{\beta(r)}\, d r}\right)=\frac{\alpha(\tau)F(\tau,y(\tau))}{\beta(\tau)}e^{\int_s^\tau\frac{1}{\beta(r)}\, d r}.
		\end{equation*}
		Then one concludes by integrating between $s$ and $t$.
		
		Under the additional assumptions, the lower bound $\dot y(t)\ge 0$ is then immediate. The upper bound in \eqref{eq:bounddoty} instead can be proved as follows: we first estimate by using \eqref{eq:doty}
		\begin{align*}
			\dot y(t)&\le \dot y(s) e^{-\int_s^t\frac{1}{\beta(r)}\, dr}+\max\limits_{\tau\in[s,t]} \alpha(\tau) F(\tau,y(\tau))\int_s^t\frac{{\rm d}}{{\rm d}\tau} e^{-\int_\tau^t\frac{1}{\beta(r)}\, dr}\, d\tau\\
			&\le \dot y(s) e^{-\int_s^t\frac{1}{\beta(r)}\, dr}+\max\limits_{\tau\in[a,t]} \alpha(\tau) F(\tau,y(\tau))\left(1- e^{-\int_s^t\frac{1}{\beta(r)}\, dr}\right)\\
			&\le \dot y(s)\vee\left(\max\limits_{\tau\in[a,t]} \alpha(\tau) F(\tau,y(\tau))\right),
		\end{align*}
		and then we send $s\searrow a$.
	\end{proof}
	
	Next lemma concerns the behaviour of the derivative $\dot y$ near the boundary points $a$ and $b$ whenever $\beta$ vanishes at some of them. As the Reader may notice comparing equation \eqref{eq:e1} with \eqref{eq:vb} and \eqref{eq:va}, the result is not surprising. In fact, it would be trivial if one had some a priori control on the second derivative $\ddot y$. The power of the lemma, which exploits formula \eqref{eq:doty}, consists indeed in the validity of the result without assuming any bound on the second derivative.
	
	\begin{lemma}\label{le:lemmadotcont}
		Let $y\in C^2(a,b)\cap C^0([a,b])$ be a solution to \eqref{eq:e1}, and suppose in addition that $\alpha$ and $F$ are nonnegative.
		\begin{subequations}
			If $\beta(b)=0$ and $\beta$ is Lipschitz continuous in a left neighborhood of $b$, then there holds 
			\begin{equation}\label{eq:vb}
				\lim_{t\to b^-}\dot{y}(t)=\alpha(b)F(b,y(b)).
			\end{equation}
			
			If $\beta(a)=0$ and if in a right neighborhood of $a$ both $\beta$ is Lipschitz continuous and $\dot{y}$ is bounded, then there also holds
			\begin{equation}\label{eq:va}
				\lim_{t\to a^+}\dot{y}(t)=\alpha(a)F(a,y(a)).
			\end{equation}  
		\end{subequations}
	\end{lemma}
	
	\begin{proof}
		
		We start showing \eqref{eq:vb}. We fix $\eps>0$, and by continuity let $(b_\eps,b)$ be such that
		\[
		|\alpha(t)F(t,y(t))-\alpha(b)F(b,y(b))|\le \eps,\qquad \text{ for }t\in(b_\eps,b).
		\]
		Then for all $t\in (b_\eps,b)$, by exploiting \eqref{eq:doty}, we can estimate
		\begin{align*}
			&|\dot y(t)-\alpha(b)F(b,y(b))|\\
			=&\left|\Big(\dot y(b_\eps)-\alpha(b)F(b,y(b))\Big)e^{-\int_{b_\eps}^t\frac{1}{\beta(r)}\, dr}+\int_{b_\eps}^t\Big(\alpha(\tau)F(\tau,y(\tau))-\alpha(b)F(b,y(b))\Big)\frac{{\rm d}}{{\rm d}\tau}e^{-\int_{\tau}^t\frac{1}{\beta(r)}\, dr}\, d\tau\right|\\
			\le &\left|\Big(\dot y(b_\eps)-\alpha(b)F(b,y(b)\Big)\right|e^{-\int_{b_\eps}^t\frac{1}{\beta(r)}\, dr}+\eps\left(1- e^{-\int_{b_\eps}^t\frac{1}{\beta(r)}\, dr}\right).
		\end{align*}
		
		We now observe that $\beta^{-1}$ is not summable near $b$ since $\beta$ is Lipschitz continuous in a neighborhood of $b$. Sending $t\to b^-$, we thus deduce
		\[
		\limsup\limits_{t\to b^-}|\dot y(t)-\alpha(b)F(b,y(b))|\le \eps,
		\]
		whence \eqref{eq:vb} follows by the arbitrariness of $\eps$.
		
		In order to show \eqref{eq:va}, we begin by sending $s\to a^+$ in \eqref{eq:doty}, obtaining (recall that now both $\beta^{-1}$ is not summable and $\dot y$ is bounded near $a$)
		\begin{equation}\label{eq:newdoty}
			\dot y(t)=\int_a^t \alpha(\tau)F(\tau,y(\tau))\frac{{\rm d}}{{\rm d}\tau}e^{-\int_{\tau}^t\frac{1}{\beta(r)}\, dr} \, d\tau, \qquad\text{for all }t\in(a,b).
		\end{equation}
		
		We now fix $\eps>0$, and by continuity let $(a,a_\eps)$ be such that
		\[
		|\alpha(t)F(t,y(t))-\alpha(a)F(a,y(a))|\le \eps,\qquad \text{ for }t\in(a,a_\eps).
		\]
		By using \eqref{eq:newdoty} and the fact that $e^{- \int_a^t \frac{1}{\beta(r)}\,dr} = 0$, for all $t\in(a,a_\eps)$ we now infer
		\begin{align*}
			&|\dot y(t)-\alpha(a)F(a,y(a))|=\left|\int_{a}^t\Big(\alpha(\tau)F(\tau,y(\tau))-\alpha(a)F(a,y(a))\Big)\frac{{\rm d}}{{\rm d}\tau}e^{-\int_{\tau}^t\frac{1}{\beta(r)}\, dr}\, d\tau\right|  \le\eps,
		\end{align*}
		and so also \eqref{eq:va} is proved.
	\end{proof}
	
	We conclude this section by showing how a function of the form $\tau\mapsto\frac{1}{\beta(\tau)}e^{-\int_\tau^t\frac{1}{\beta(r)}\,d r}$, appearing in \eqref{eq:doty}, is an approximation of a Dirac delta centered at $t$, as soon as $\beta$ is small enough. 
	\begin{lemma}\label{lemma:apprdelta}
		Let $\{f_n\}$ and $\{g_n\}$ be sequences of continuous functions in $[a,b]$. Assume that $f_n$ uniformly converges to $f$ in $[a,b]$ as $n\to \infty$, and that $g_n$ is positive and vanishes uniformly in $[a,b]$ as $n\to \infty$. Then for all $t\in (a,b]$ there holds
		\begin{equation*}
			\lim\limits_{n\to \infty}\int_a^t \frac{f_n(\tau)}{g_n(\tau)}e^{-\int_\tau^t\frac{1}{g_n(r)}\,d r}\,d\tau=f(t).
		\end{equation*}
	\end{lemma}
	\begin{proof}
		We begin by observing that for all $t\in (a,b]$ the following facts are true:
		\begin{itemize}
			\item[(a)] $\displaystyle \int_a^t\frac{1}{g_n(\tau)}e^{-\int_\tau^t\frac{1}{g_n(r)}\,d r}\, d\tau=1- e^{-\int_a^t\frac{1}{g_n(r)}\,d r}$;
			\item[(b)]$\displaystyle \lim\limits_{n\to \infty}e^{-\int_a^t\frac{1}{g_n(r)}\,d r}=0 $;
			\item[(c)] $\displaystyle \lim\limits_{n\to \infty}\int_a^te^{-\int_\tau^t\frac{1}{g_n(r)}\,d r}\, d\tau=0 $.
		\end{itemize}
		Identity (a) is simply the Fundamental Theorem of Calculus, while the limits in (b) and (c) can be proved arguing as follows. Since $g_n$ vanishes uniformly, for any $\eps>0$ we definitively have $\|g_n\|_{C^0([a,b])}\le \eps$, whence
		\begin{equation*}
			e^{-\int_a^t\frac{1}{g_n(r)}\,d r}\le e^{-\frac{t-a}{\eps}},\qquad\text{and}\qquad \int_a^te^{-\int_\tau^t\frac{1}{g_n(r)}\,d r}\, d\tau\le \int_a^te^{-\frac{t-\tau}{\eps}}\, d\tau=\eps\left(1-e^{-\frac{t-a}{\eps}}\right).
		\end{equation*}
		Sending first $n\to \infty$ and then $\eps\searrow 0$ we obtain (b) and (c).
		
		We now fix $t\in (a,b]$, and by means of (a) we estimate
		\begin{align}
			&\quad\, \left|\int_a^t \frac{f_n(\tau)}{g_n(\tau)}e^{-\int_\tau^t\frac{1}{g_n(r)}\,d r}\,d\tau-f(t)\right|\nonumber\\
			&\le \int_a^t|f_n(\tau)-f(\tau)|\frac{1}{g_n(\tau)}e^{-\int_\tau^t\frac{1}{g_n(r)}\,d r}\,d\tau+\left|\int_a^t (f(\tau)-f(t))\frac{1}{g_n(\tau)}e^{-\int_\tau^t\frac{1}{g_n(r)}\,d r}\,d\tau\right|+|f(t)|e^{-\int_a^t\frac{1}{g_n(r)}\,d r}\nonumber\\
			&\le \|f_n-f\|_{C^0([a,b])}+|f(t)|e^{-\int_a^t\frac{1}{g_n(r)}\,d r}+\left|\int_a^t (f(\tau)-f(t))\frac{1}{g_n(\tau)}e^{-\int_\tau^t\frac{1}{g_n(r)}\,d r}\,d\tau\right|.\label{eq:3term}
		\end{align}
		The first two terms in the last line above vanish as $n\to \infty$ by assumption and by (b), respectively; so we conclude if we prove that also the third term vanishes.
		
		For the sake of clarity we set $h(\tau):= f(\tau)-f(t)$, and we consider a sequence $\{h_m\}$ of smooth functions uniformly converging to $h$ in $[a,t]$. Denoting by $J_n(t)$ the last term in \eqref{eq:3term}, integrating by parts we infer
		\begin{align*}
			J_n(t)&\le\int_a^t |h(\tau)-h_m(\tau)|\frac{1}{g_n(\tau)}e^{-\int_\tau^t\frac{1}{g_n(r)}\,d r}\,d\tau+\left|\int_a^t h_m(\tau)\frac{1}{g_n(\tau)}e^{-\int_\tau^t\frac{1}{g_n(r)}\,d r}\, d\tau\right|\\
			&\le \|h-h_m\|_{C^0([a,t])}+\left|h_m(t)-h_m(a)e^{-\int_a^t\frac{1}{g_n(r)}\,d r}-\int_a^t \dot h_m(\tau)e^{-\int_\tau^t\frac{1}{g_n(r)}\,d r}\, d\tau\right|\\
			&\le \|h-h_m\|_{C^0([a,t])}+|h_m(t)|+|h_m(a)|e^{-\int_a^t\frac{1}{g_n(r)}\,d r}+\|\dot h_m\|_{C^0([a,t])}\int_a^te^{-\int_\tau^t\frac{1}{g_n(r)}\,d r}\, d\tau.
		\end{align*}
		Letting $n\to \infty$, by using (b) and (c) we deduce
		\begin{equation*}
			\limsup\limits_{n\to \infty} J_n(t)\le \|h-h_m\|_{C^0([a,t])}+|h_m(t)|,
		\end{equation*}
		and sending then $m\to \infty$ we conclude since $h(t)=f(t)-f(t)=0$.
	\end{proof}
	
	\section{Existence of solutions to the discrete system}\label{sec:existdiscrete}
	
	This section is devoted to the proofs of Theorem~\ref{thm:maindiscrete} and Proposition~\ref{prop:nosaturated}.
	
	\begin{proof}[Proof of Theorem~\ref{thm:maindiscrete}]
		The proof is entirely based on the following inductive step.
		
		\emph{If $x_{i+1}\in C^0([0,T])$ is nondecreasing and fulfils $x_{i+1}(0)=x_{i+1}^0$, then there exists $x_i\in C^{0,1}([0,T])$ solution to \eqref{eq:system} in the sense of Definition~\ref{def:sol} satisfying \eqref{eq:uniflipbound}.
			Assuming additionally that $F$ is positive, if $x_{i+1}$ is strictly increasing, then $x_i$ is strictly increasing as well.
		}
		
		It is clear that the inductive step above entails the claim of the theorem. For this reason, what follows is devoted to its proof, which we split in several steps for simplicity.
		
		\textit{\textbf{Step 1. Approximated problem and compactness.}}
		
		\noindent For $\delta>0$ we consider the approximated problem 
		\begin{equation}\label{eq:equazionedelta}
			\begin{cases}
				\eps(\zeta(\rho^\delta_i(t))+\delta)\ddot{x}^\delta_i(t)+\gamma\dot{x}^\delta_i(t)=\vartheta(\rho^\delta_i(t))F(t,x_i^\delta(t)),\qquad \text{for }t\in (0,T),\\
				x_i^\delta(0)=x_i^0,\qquad \dot{x}_i^\delta(0)=\widetilde v_i^0,
			\end{cases}
		\end{equation}
		where $\widetilde v_i^0$ has been introduced in \eqref{eq:vtilde}, while we define $\rho_i^\delta(t):=(Nd_i^\delta(t))^{-1}$ and $d_i^\delta(t)=x_{i+1}(t)-x_i^\delta(t)$. 
		Noting that in~\eqref{eq:equazionedelta} the coefficient of the second derivative is no more degenerate thanks to the addition of the parameter $\delta$, by the Peano Theorem there exists a local solution $x_i^\delta\in C^2([0,\widehat T_\delta])$ for some $\widehat T_\delta>0$, which we will show to be independent of $\delta$.
		
		Without loss of generality, we can assume that $\widehat T_\delta$ is the first time at which $x_i^\delta(\widehat T_\delta)=\frac{x_{i+1}^0+x_i^0}{2}$; indeed, if $x_i^\delta(t)<\frac{x_{i+1}^0+x_i^0}{2}$ for all the evolution, then $x_i^\delta$ exists in the whole $[0,T]$.
		
		By Lemma~\ref{le:lemmadoty} (with $\alpha(t)=\frac 1\gamma \vartheta(\rho_i^\delta(t))$ and $\beta(t)=\frac{\eps}{\gamma}(\zeta(\rho^\delta_i(t))+\delta)$) we deduce that
		\begin{equation*}
			0\le\dot x_i^\delta(t)\le \widetilde v_i^0\vee \left(\frac{\bar\vartheta}{\gamma}\max\limits_{[0,T]\times[x_0^0,x_N(T)]}F\right),\qquad\text{for all }t\in [0,\widehat T_\delta].
		\end{equation*}
		Since from the above estimate the velocity $\dot{x}_i^\delta$ is bounded by a constant independent of $\delta$, and since clearly the distance $x_{i+1}^0-x_i^0$ does not depend on $\delta$, we deduce that $\widehat T_\delta$ does not depend on $\delta$ as well.
		
		Summarising, we have proved that there exists $\widehat T>0$ such that $x_i^\delta\in C^2([0,\widehat T])$ satisfies 
		\begin{equation}\label{eq:ub}
			x_i^0\leq x_i^\delta(t)\leq \frac{x_{i+1}^0+x_i^0}{2},\qquad 0\leq \dot{x}_i^\delta(t)\leq\widetilde v_i^0\vee \left(\frac{\bar\vartheta}{\gamma}\max\limits_{[0,T]\times[x_0^0,x_N(T)]}F\right),\qquad \text{for all }t\in[0,\widehat T].
		\end{equation}
		
		Hence, by Ascoli-Arzel\`a Theorem, we have that, up to passing to a subsequence, $x_i^\delta\rightarrow x_i$ uniformly in $[0,\widehat T]$ as $\delta\to 0$, for a certain nondecreasing function $x_i\in C^{0,1}([0,\widehat T])$, which satisfies
		\begin{equation*}
			0\leq \dot{x}_i(t)\leq\widetilde v_i^0\vee \left(\frac{\bar\vartheta}{\gamma}\max\limits_{[0,T]\times[x_0^0,x_N(T)]}F\right),\qquad \text{for almost every }t\in[0,\widehat T],
		\end{equation*}
		attains the initial condition $x_i(0)=x_i^0$, and fulfils $x_i(t)<x_{i+1}(t)$ in $[0,\widehat T]$.
		Moreover, the fact that the distance between $x_i^\delta$ and $x_{i+1}^\delta$ is uniformly bounded from below and far from $0$ entails a uniform bounds on the densities and, hence, the uniform convergence $\rho_i^\delta \to \rho_i$ in $[0,\hat{T}]$.
		
		\textit{\textbf{Step 2. Existence of a local solution $x_i$.}}
		
		\noindent We now aim to show that the limit function $x_i$ is a local solution of the microscopic traffic model. For the moment, we have seen that points~\ref{hi},~\ref{hii} of Definition~\ref{def:sol} are fulfilled by $x_i$ (in $[0,\widehat T]$). We then need to show that 
		\begin{equation}\label{eq:claim}
			\rho_i(t)\leq \bar \rho_\vartheta,\qquad \text{for all }t\in[0,\widehat T].
		\end{equation}
		For the sake of contradiction, if the claim above is false, then there exists $\overline t \in (0,\widehat T]$ such that $\rho_i(\bar t)>\bar \rho_\vartheta$ (i.e. $d_i(\overline t)<(N\bar \rho_\vartheta)^{-1}$).
		By continuity, there exists $a\in [0,\overline t)$ such that $\rho_i(a)=\bar \rho_\vartheta$ and $\rho_i(t)>\bar \rho_\vartheta$ in $(a,\overline t]$.
		We consider now a sequence $\{a_k\}\subseteq(a,\bar t)$ converging to $a$ as $k\to \infty$. Observe that $\rho_i^\delta(t)\geq \bar \rho_\vartheta$ definitively in $[a_k,\overline t]$, thanks to the uniform convergence of $\rho_i^\delta$ to $\rho_i$ in $[0,\widehat T]$. 
		
		In particular, $\zeta(\rho_i^\delta(t))=\vartheta(\rho_i^\delta(t))=0$ for $t\in [a_k,\overline t]$, and so by \eqref{eq:doty} we obtain
		\begin{equation*}
			\dot x_i^\delta(t)=\dot x_i^\delta(a_k)e^{-\frac{\gamma}{\eps}\frac{t-a_k}{\delta}},\qquad\text{for all }t\in[a_k,\bar t],
		\end{equation*}
		whence $\lim\limits_{\delta\to 0}\dot x_i^\delta(t)=0$ for all $t\in(a_k,\bar t]$.
		
		Since $x_{i+1}$ is nondecreasing, we have \[
		d_i^\delta(\overline t)=x_{i+1}(\bar t)-x_i^\delta(\bar t)\ge x_{i+1}(a_k)-x_i^\delta(a_k)-\int_{a_k}^{\overline t}\dot{x}_i^\delta(\tau)\,d\tau =d_i^\delta(a_k)-\int_{a_k}^{\overline t}\dot{x}_i^\delta(\tau)\,d\tau.
		\]
		Passing to the limit as $\delta\to 0$, by Dominated Convergence Theorem we thus deduce that $d_i(\overline t)\geq d_i(a_k)$ and eventually passing to the limit as $k\to \infty$ (so that $a_k\to a$), we obtain\[
		d_i(\overline t)\geq d_i(a)=\frac{1}{N\rho_i(a)}=\frac{1}{N\bar \rho_\vartheta},
		\]
		which contradicts the fact that $d_i(\overline t)<(N\bar \rho_\vartheta)^{-1}$. Therefore the claim~\eqref{eq:claim} is proved, and so $x_i$ satisfies point~\ref{hiii} of Definition~\ref{def:sol}, too.
		
		Let us now focus on proving points~\ref{hiv} and~\ref{hvi}.
		For all $k\in \N$, we set \[
		A_i^k:=\Big\{t\in [0,\widehat T] : \zeta(\rho_i(t))>\frac1k\Big\},\qquad \text{so that }\qquad \bigcup_{k\in\N}A_i^k=\Big\{t\in [0,\widehat T] : \zeta(\rho_i(t))>0\Big\},
		\]
		namely the set in \eqref{eq:open1}. Clearly, it is enough to prove that $x_i$ solves equation~\eqref{eq:systempos} in each connected component $I_j$ of $A_i^k$, where for the sake of clarity we omit the dependence on $i,k$.
		By uniform convergence, one has $\zeta(\rho_i^\delta)\geq\tfrac{1}{2k}$ in $I_j$ for $\delta$ sufficiently small, thus from the equation we deduce, recalling~\eqref{eq:ub}, that \[
		\eps|\ddot{x}_i^\delta(t)|\le \frac{\bar\vartheta \max\limits_{[0,T]\times[x_0^0,x_N(T)]}F}{\frac{1}{2k}}+\gamma\frac{ \dot{x}_i^\delta(t)}{\frac{1}{2k}}\leq 4k \left[ (\gamma v_i^0)\vee \left({\bar\vartheta}\max\limits_{[0,T]\times[x_0^0,x_N(T)]}F\right)\right],\qquad \text{in }I_j.
		\]
		
		By Ascoli-Arzel\`a Theorem, we now obtain that \[
		\dot{x}_i^\delta\rightarrow \dot{x}_i,\qquad \text{uniformly in $\overline I_j$, as $\delta\to 0$},
		\]
		so that $\dot x_i\in C^{0,1}(\overline I_j)$ and condition~\ref{hvi} is automatically fulfilled by $x_i$, since for unsaturated initial positions one has $0\in A_i^k$ for some $k$.
		
		Now, for all $\varphi\in C^\infty_{\rm c}(\mathring{I}_j)$, recalling that $\zeta (\rho_i^\delta)\geq \tfrac{1}{2k}$ in $I_j$ (for $\delta$ sufficiently small), we can test equation~\eqref{eq:equazionedelta} with $\varphi$, divide by $\zeta(\rho^\delta_i(\tau))+\delta$ and integrate over $I_j$, obtaining\[
		\gamma \int_{I_j}\frac{\dot{x}^\delta_i(\tau)}{\zeta(\rho^\delta_i(\tau))+\delta}\varphi(\tau)\,d\tau=\int_{I_j}\left(\eps\dot{x}^\delta_i(\tau)\dot\varphi(\tau)+\frac{\vartheta(\rho^\delta_i(\tau))}{\zeta(\rho^\delta_i(\tau))+\delta}F(\tau,x_i^\delta(\tau))\varphi(\tau)\right)\,d\tau.
		\]
		Letting $\delta\to 0$, we deduce (recalling that $\zeta(\rho_i)>\tfrac1k$ in $I_j$) 
		\begin{equation*}
			\gamma\int_{I_j}\frac{ \dot{x}_i(\tau)}{\zeta(\rho_i(\tau))}\varphi(\tau)\,d\tau=\int_{I_j}\left(\eps\dot{x}_i(\tau)\dot \varphi(\tau)+\frac{\vartheta(\rho_i(\tau))}{\zeta(\rho_i(\tau))}F(\tau,x_i(\tau))\varphi(\tau)\right)\,d\tau,
		\end{equation*}
		whence, by the arbitrariness of $\varphi$, we infer that $x_i\in C^2(\mathring I_j)$ and that~\ref{hiv} is proved.
		
		It remains only to show that $x_i$ also satisfies condition~\ref{hv}. Observe that it is enough to prove the statement for any open interval $(a,b)$ contained in \eqref{eq:open2}. Since $\rho_i^\delta$ uniformly converges to $\rho_i$ in the whole $[0,\widehat T]$ as $\delta\to 0$, in particular we have
		\begin{equation*}
			\zeta(\rho_i^\delta)\to \zeta(\rho_i)\equiv0,\qquad \text{uniformly in $[a,b]$ as $\delta\to 0$}.
		\end{equation*}
		
		We now recall that formula \eqref{eq:doty} yields
		\begin{equation*}
			\gamma\dot x_i^\delta(t)=\gamma \dot x_i^\delta(a)e^{-\frac\gamma\eps\int_a^t\frac{1}{\zeta(\rho_i^\delta(r))+\delta}\, dr}+\frac \gamma\eps\int_a^t \frac{\vartheta(\rho_i^\delta(\tau))F(\tau,x_i^\delta(\tau))}{\zeta(\rho_i^\delta(\tau))+\delta}e^{-\frac\gamma\eps\int_\tau^t\frac{1}{\zeta(\rho_i^\delta(r))+\delta}\, dr}\, d\tau,\quad\text{for }t\in(a,b).
		\end{equation*}
		So, by exploiting Lemma~\ref{lemma:apprdelta} (see also (b) within its proof) with $f_\delta(t)=\vartheta(\rho_i^\delta(t))F(t,x_i^\delta(t))$ and $g_\delta(t)=\frac \eps\gamma (\zeta(\rho_i^\delta(t))+\delta)$, we deduce that
		\begin{equation*}
			\lim\limits_{\delta\to 0}\gamma \dot x_i^\delta(t)=\vartheta(\rho_i(t))F(t,x_i(t)), \qquad\text{for all }t\in (a,b).
		\end{equation*}
		Fixing $\varphi\in C^\infty_{\rm c}(a,b)$, the above limit implies
		\begin{equation*}
			-\gamma\int_a^b\! x_i(\tau)\dot \varphi(\tau) \,d\tau=\lim\limits_{\delta\to 0}-\gamma\int_a^b\! x_i^\delta(\tau)\dot \varphi(\tau) \,d\tau=\lim\limits_{\delta\to 0}\gamma\int_a^b\! \dot x_i^\delta(\tau)\varphi(\tau) \,d\tau=\int_a^b \!\vartheta(\rho_i(\tau))F(\tau,x_i(\tau)) \varphi(\tau) \,d\tau,
		\end{equation*}
		whence we obtain that $x_i\in C^1(a,b)$ and that equation \eqref{eq:systempos1ord} is fulfilled in $(a,b)$. This yields \ref{hv} and so $x_i$ is a local solution in the sense of Definition~\ref{def:sol}.
		
		\textit{\textbf{Step 3. Extension of the local solution to the whole interval $[0,T]$.}}
		
		\noindent Let $\overline T$ be the maximal time of existence of $x_i$, namely
		\begin{equation}\label{eq:T}
			\overline T=\sup\left\{t\in(0,T] : \begin{gathered}
				x_i\in C^{0,1}([0,t]) \text{ satisfies conditions~\ref{hi}-\ref{hvi} in $[0,t]$,}\\
				\text{and }\dot{x}_i(\tau)\leq\widetilde v_i^0\vee \left(\frac{\bar\vartheta}{\gamma}\max\limits_{[0,T]\times[x_0^0,x_N(T)]}F\right) \text{for almost every $\tau\in[0,t]$}
			\end{gathered}\right\}.
		\end{equation}
		
		We argue for the sake of contradiction, and assume that $\overline T<T$.
		Since $x_{i+1}\in C^0([0,T])$ by inductive assumption, by conditions~\ref{hii} and~\ref{hiii} the following limit exists by monotonicity and is finite:\[
		x_i(\overline T)=\lim_{t\to \overline T^-}x_i(t).
		\]
		Moreover, by~\ref{hiii}, we also deduce that $\rho_i(\overline T)\leq \bar \rho_\vartheta$, so that $x_i(\overline T)<x_{i+1}(\overline T)$. By the uniform bound on the derivative, we finally infer that $x_i\in C^{0,1}([0,\overline T])$ with $\dot{x}_i(\tau)\leq\widetilde v_i^0\vee \left(\frac{\bar\vartheta}{\gamma}\max\limits_{[0,T]\times[x_0^0,x_N(T)]}F\right)$ almost everywhere in $[0,\overline T]$.
		
		If $\rho_i(\overline T)\in[\bar\rho_\zeta,\bar\rho_\vartheta]$, we do not need to give a velocity for $x_i$ at $\overline T$. On the other hand, if $\rho_i(\overline T)<\bar \rho_\zeta$, by continuity it must exist $\widetilde \rho\in (0,\bar \rho_\zeta)$ and $\widetilde T\in(0,\overline T)$ such that $\rho_i(t)\leq \widetilde \rho<\bar \rho_\zeta$ for all $t\in [\widetilde T,\overline T]$.
		Thus, there exists $\widetilde\zeta>0$ such that $\zeta(\rho_i(t))\geq \widetilde \zeta>0$ for all $t\in [\widetilde T,\overline T]$
		and so $x_i$ solves the equation\[
		\eps\zeta(\rho_i(t))\ddot{x}_i(t)+\gamma\dot{x}_i(t)=\vartheta(\rho_i(t))F(t,x_i(t)),\qquad \text{in $(\widetilde T,\overline T)$.} 
		\]
		Since $\zeta(\rho_i(t))\geq \widetilde \zeta>0$ for all $t\in [\widetilde T,\overline T]$, we also deduce that the solution $x_i$ is actually defined and of class $C^2$ in $[\widetilde T,\overline T]$, whence $\dot{x}_i(\overline T)$ exists and is (necessarily) nonnegative and satisfies the bound in~\eqref{eq:T}.
		
		In both the saturated and unsaturated cases it is then possible to check that $x_i\in C^{0,1}([0,\overline T])$ satisfies all conditions~\ref{hi}-\ref{hvi} in $[0,\overline T]$ and thus $\overline T$ actually realizes the maximum in \eqref{eq:T}.
		
		We can now repeat the local existence argument of Steps 1 and 2 with $\overline T$ as starting time, thus finding a solution $\widecheck x_i\in C^{0,1}([\overline T, \widecheck T])$ with initial data $x_i(\overline T)$ and $\dot x_i(\overline T)$, fulfilling conditions~\ref{hi}-\ref{hvi} in $[\overline T,\widecheck T]$ for some $\widecheck T\in(\overline T,T]$ and satisfying
		\begin{equation*}
			\dot{\widecheck x}_i(\tau)\leq{\dot{x}}_i(\overline T)\vee \left(\frac{\bar\vartheta}{\gamma}\max\limits_{[0,T]\times[x_0^0,x_N(T)]}F\right)\le \widetilde v_i^0\vee \left(\frac{\bar\vartheta}{\gamma}\max\limits_{[0,T]\times[x_0^0,x_N(T)]}F\right),\qquad\text{for almost every }\tau\in(\overline T, \widecheck T).
		\end{equation*} 
		By gluing $x_i$ and $\widecheck x_i$ we finally obtain a new Lipschitz continuous function satisfying conditions~\ref{hi}-\ref{hvi} in the whole $[0,\widecheck T]$ and with Lipschitz constant bounded by $\widetilde v_i^0\vee \left(\frac{\bar\vartheta}{\gamma}\max\limits_{[0,T]\times[x_0^0,x_N(T)]}F\right)$. We thus contradict the maximality of $\overline T$, and we conclude Step 3.
		
		\textit{\textbf{Step 4. If $F$ is positive and $x_{i+1}$ is strictly increasing, then $x_i$ is strictly increasing.}}
		
		\noindent Assume by contradiction that there exist $0\le a<b\le T$ such that $x_i(a)=x_i(b)$. Since $x_i$ is nondecreasing, this means that $x_i$ is constant in $[a,b]$. Observe that the interval $(a,b)$ can not intersect the open set $\{\rho_i<\bar\rho_\zeta\}\cup \{\bar\rho_\zeta<\rho_i<\bar\rho_\vartheta\}$, indeed constant functions do not solve equation \eqref{eq:systempos} nor equation \eqref{eq:systempos1ord} due to the positivity of $F$. We thus infer that $(a,b)$ is contained either in the set $\{\rho_i=\bar\rho_\zeta\}$ or in the set $\{\rho_i=\bar\rho_\vartheta\}$. In both cases, by the definition \eqref{eq:rho} of $\rho_i$, one deduces that $x_{i+1}$ is constant in $(a,b)$. We thus contradict the inductive assumption, and we conclude the proof of Step 4 and of the whole theorem.
		
	\end{proof}
	
	We now move to the proof of Proposition~\ref{prop:nosaturated}.
	
	\begin{proof}[Proof of Proposition~\ref{prop:nosaturated}]
		Let $x$ be a solution to the microscopic traffic model fulfilling in addition \eqref{eq:isolatedpoints}. As in the proof of Theorem~\ref{thm:maindiscrete} we argue by induction, showing the following:
		
		\emph{If $x_{i+1}\in C^1([0,T])$ satisfies $\dot x_{i+1}(t)>0$ for all $t\in(0,T]$, then $x_i$ satisfies
			\begin{itemize}
				\item $x_i$ belongs to $C^1([0,T])$;
				\item $\rho_i(t)<\bar \rho_\vartheta$ and $\dot{x}_i(t)>0$ for all $t\in (0,T]$;
				\item $\dot{x}_i(0)=\frac 1\gamma\vartheta(\rho_i^0)F(0,x_i^0)$ if $i\in \Sigma^0$;
				\item if $\bar\rho_\zeta=\bar\rho_\vartheta$, then $x_i$ belongs to $C^2((0,T])$.
			\end{itemize}
		}
		\textit{\textbf{Step 1. $\rho_i(t)<\bar\rho_\vartheta$ for all $t\in(0,T]$.}}
		
		\noindent First of all, let us prove that the set $\{t\in(0,T] : \rho_i(t)<\bar \rho_\vartheta\}$ is nonempty and $0$ is an accumulation point for it. By contradiction, assume there exists an open interval $(0,s)\subseteq \{t\in(0,T] : \rho_i(t)=\bar \rho_\vartheta\}$. Thus, by condition \ref{hv} in Definition~\ref{def:sol}, we infer that $x_i\in C^1(0,s)$ satisfies
		\begin{equation*}
			\gamma\dot x_i(t)=\vartheta (\rho_i(t))F(t,x_i(t))=\vartheta (\bar\rho_\vartheta)F(t,x_i(t))=0,\qquad \text{for }t\in (0,s),
		\end{equation*}
		namely $x_i(t)\equiv x_i^0$ for all $t\in [0,s]$. By definition \eqref{eq:rho} of $\rho_i$, we hence deduce that
		\begin{equation*}
			x_{i+1}(t)=x_i^0+\frac{1}{N\bar\rho_\vartheta},\qquad\text{for }t\in (0,s),
		\end{equation*}
		which contradicts the inductive assumption $\dot x_{i+1}(t)>0$ for all $t\in (0,T]$.
		
		We now prove the claim of Step 1. Assume by contradiction that there exists $b\in (0,T]$ such that $\rho_i(b)=\bar\rho_\vartheta$. Since we proved that $0$ is an accumulation point for the set $\{t\in(0,T] : \rho_i(t)<\bar \rho_\vartheta\}$, there also exists $a\in (0,b)$ satisfying $\rho_i(a)<\bar\rho_\vartheta$. Without loss of generality, we may assume that $b$ is the first point greater than $a$ at which $\rho_i$ reaches the threshold $\bar\rho_\vartheta$, namely $\rho_i(t)<\bar\rho_\vartheta$ for all $t\in (a,b)$ and $\rho_i(b)=\bar\rho_\vartheta$. Equivalently, we may write
		\begin{equation}\label{eq:distb}
			d_i(t)>\frac{1}{N\bar\rho_\vartheta},\quad\text{for all }t\in (a,b),\qquad\text{ and }\qquad  d_i(b)=\frac{1}{N\bar\rho_\vartheta}.
		\end{equation}
		
		If $\bar\rho_\zeta<\bar\rho_\vartheta$, up to possibly choosing a larger $a$, we can assume without loss of generality that $\bar{\rho}_\zeta <  \rho_i(t) < \bar{\rho}_\vartheta$ in $(a,b)$, and so by condition \ref{hv} we know that $x_i\in C^1(a,b)$ satisfies
		\begin{equation*}
			\gamma\dot x_i(t)=\vartheta (\rho_i(t))F(t,x_i(t)),\quad\text{for }t\in (a,b).
		\end{equation*}
		In particular, from the continuity of $\vartheta$ one deduces 
		\begin{equation}\label{eq:derzero}
			\lim_{t\to b^-}\dot{x}_i(t)=0.
		\end{equation}
		
		We now show that we reach the same conclusion also in the case $\bar\rho_\zeta=\bar\rho_\vartheta$. Indeed, by condition \ref{hiv} of Definition~\ref{def:sol}, we know that $x_i\in C^2(a,b)$ satisfies
		\begin{equation*}
			\eps\zeta(\rho_i(t))\ddot x_i(t)+\gamma\dot x_i(t)=\vartheta (\rho_i(t))F(t,x_i(t)),\quad\text{for }t\in (a,b).
		\end{equation*}
		Since $\zeta$ is Lipschitz continuous near $\bar\rho_\zeta$, $x_{i+1}$ is of class $C^1([0,T])$, $x_i$ is Lipschitz continuous in the whole $[0,T]$ and $\rho_i(t)\le \bar\rho_\zeta$, we deduce that the map $t\mapsto \zeta(\rho_i(t))$ is Lipschitz continuous near $b$. We are thus in position to apply Lemma~\ref{le:lemmadotcont} and infer that \eqref{eq:derzero} holds true also in this case.
		
		Summing up, the limit in \eqref{eq:derzero} yields that $\dot{x}_i$ is differentiable from the left in $b$ and $\dot{x}_i^-(b)=0$. In particular, also $d_i$ is differentiable from the left in $b$ and it holds
		\begin{equation}\label{eq:distance}
			\dot{d}_i^-(b)=\dot{x}_{i+1}(b)-\dot{x}_i^-(b)=\dot{x}_{i+1}(b).
		\end{equation}
		
		This leads to a contradiction. Indeed, by \eqref{eq:distb} it must hold $\dot{d}_i^-(b)\leq 0$, which contradicts the inductive assumption on $x_{i+1}$, namely $\dot{x}_{i+1}(b)>0$, by using \eqref{eq:distance}.
		
		Step 1 is thus concluded.
		
		\textit{\textbf{Step 2. $x_i\in C^1([0,T])$ satisfies $\dot{x}_i(t)>0$ for all $t\in (0,T]$.}}
		
		\noindent We distinguish between the cases $\bar\rho_\zeta<\bar\rho_\vartheta$ and $\bar\rho_\zeta=\bar\rho_\vartheta$.
		
		We begin with the former case. Fix $\bar t\in (0,T]$. If $\bar t$ belongs to the set \eqref{eq:open1}, then $x_i$ is differentiable at $\bar t$ and $\dot{x}_i(\bar t)>0$ by the expression \eqref{eq:doty}. If instead $\bar t$ belongs to the set \eqref{eq:open2}, again $x_i$ is differentiable at $\bar t$ and the property $\dot{x}_i(\bar t)>0$ directly follows from the equation \eqref{eq:systempos1ord} (we recall that by Step 1 we know that $\rho_i(\bar t)<\bar\rho_\vartheta$). We just need to check the case $\bar t\in \partial \{t\in [0,T] : \rho_i(t)={{\overline{\rho}}_\zeta}\}$. Since by \eqref{eq:isolatedpoints} that set consists of isolated points, we can find $a<\bar t<b$ such that the intervals $(a,\bar t)$ and $(\bar t, b)$ are contained in one of the open sets in \eqref{eq:open1} or \eqref{eq:open2}. By using directly the first order equation \eqref{eq:systempos1ord}, or by means of Lemma~\ref{le:lemmadotcont}, we deduce that in any case the derivative $\dot x_i(\bar t)$ exists and it holds $\gamma\dot x_i(\bar t)=\vartheta(\bar\rho_\zeta)F(\bar t,x_i(\bar t))>0$. A similar argument shows that $x_i$ is differentiable (from the right) also in $0$ and that $x_i$ is of class $C^1([0,T])$.
		
		The case $\bar\rho_\zeta=\bar\rho_\vartheta$ is simpler. Indeed, by Step 1 we know that $(0,T]\subseteq  \{t\in [0,T] : \rho_i(t)<{{\overline{\rho}}_\zeta}\}$, and so by Definition~\ref{def:sol} we automatically have $x_i\in C^2((0,T])$, and by the expression \eqref{eq:doty} we may infer $\dot{x}_i(t)>0$ for $t\in(0,T]$. If $\rho_i^0<\bar\rho_\zeta$, the same argument yields $x_i\in C^2([0,T])$; otherwise we may use Lemma~\ref{le:lemmadotcont}, deducing that $\dot x_i(0)$ exists, it holds $\gamma\dot x_i(0)=\vartheta(\bar\rho_\vartheta)F(0 ,x_i(0))=0$, and also $x_i$ is of class $C^1([0,T])$.
		
		\textit{\textbf{Step 3. Conclusion of the proof.}}
		
		\noindent We are left to show the last two points of the inductive step. Actually, the last one have been already proved at the end of Step 2. 
		
		So, assume that $i\in \Sigma^0$. If $\rho_i^0>\bar\rho_\zeta$, then necessarily $x_i$ solves the first order equation \eqref{eq:systempos1ord} in a right neighborhood of $0$. This implies $\dot{x}_i(0)=\frac 1\gamma\vartheta(\rho_i^0)F(0,x_i^0)$. If instead $\rho_i^0=\bar\rho_\zeta$, by assumption \eqref{eq:isolatedpoints} there exists a right neighborhood of $0$ in which $x_i$ solves either the second order equation \eqref{eq:systempos} or the first order one \eqref{eq:systempos1ord}. By using Lemma~\ref{le:lemmadotcont} in the former case, or directly by the equation in the second case, one obtains $\dot{x}_i(0)=\frac 1\gamma\vartheta(\rho_i^0)F(0,x_i^0)$, and so we conclude the proof of the proposition.
	\end{proof}
	
	\section{Many-particle limit}\label{sec:manyparticle}
	
	The content of this section is the proof of Theorem~\ref{thm:maincont}, so we tacitly assume its hypotheses. To this aim, we first provide uniform estimates on the discrete interpolants \eqref{eq:definterpolants} which will also be used in the next section to prove Theorem~\ref{thm:vaninert}. We stress that we are assuming that \eqref{eq:zeta=theta} is in force, therefore \eqref{eq:barrho} applies.
	
	\begin{proposition}\label{prop:uniformbounds}
		Under the assumptions of Theorem~\ref{thm:maincont}, the following uniform bounds hold:
		\begin{itemize}
			\item $\|\rho^N\|_{L^\infty((0,T)\times \R)}\le \bar \rho$;
			\item $\|e_k^N\|_{L^\infty((0,T)\times \R)}\le C$ for $k=1,2$;
			\item $\left\|\frac{e_1^N}{\vartheta(\rho^N)}\right\|_{L^1((0,T)\times\R)}\le C$;
		\end{itemize}
		where the constant $C>0$ depends neither on $N$ nor on $\eps$.
		
		In particular, either if $\eps$ is fixed or if $\eps=\eps_N\to 0$, up to nonrelabelled subsequences, as $N\to \infty$ we infer
		\begin{subequations}
			\begin{align}
				& \rho^N\overset{*}{\rightharpoonup} \rho,&& e_k^N\overset{*}{\rightharpoonup}e_k,\qquad\,\,\text{weakly$^*$ in $L^\infty((0,T)\times\R)$, for $k=1,2$};\label{eq:lim1}\\
				&\frac{e_1^N}{\vartheta(\rho^N)}dtdx\overset{*}{\rightharpoonup} \lambda,&& \qquad\qquad\qquad\,\,\,\text{weakly$^*$ in $\mathcal M([0,T]\times\R)$}\label{eq:lim2}.
			\end{align} 
		\end{subequations}
	\end{proposition}
	
	\begin{proof}
		By using condition~\ref{hiii} in Definition~\ref{def:sol} it immediately follows that \[
		\|\rho^N\|_{L^\infty((0,T)\times \R)}\leq \bar \rho. 
		\]
		Moreover, the latter bound together with \eqref{eq:uniflipbound}, \eqref{eq:initialassumption} and the very definition of the interpolants \eqref{eq:definterpolants}, imply that
		\[
		\|e_k^N\|_{L^\infty((0,T)\times\R)}\leq C,\qquad \text{ for $k=1,2$. }
		\]
		Finally, we note that, by using also equation \eqref{eq:systempos} (recall that here $\zeta=\vartheta$), there holds \[
		\begin{split}
			\gamma \left\|\frac{e_1^N}{\vartheta(\rho^N)}\right\|_{L^1((0,T)\times\R)}&=\sum_{i=0}^{N-1}\int_0^T\int_{x_i(t)}^{x_{i+1}(t)}\gamma\frac{\rho_i(t)\dot{x}_i(t)}{\vartheta(\rho_i(t))}\,dx\,dt=\frac{1}{N}\sum_{i=0}^{N-1}\int_0^T\gamma\frac{\dot{x}_i(t)}{\vartheta(\rho_i(t))}\,dt\\
			&= \frac1N\sum_{i=0}^{N-1}\int_0^T(F(t,x_i(t))-\eps\ddot{x}_i(t))\,dt\leq {T} \max_{[0,T]\times[s,S]}F-\frac{\eps}{N}\sum_{i=0}^{N-1}(\dot{x}_i(T)-\dot{x}_i(0))\\
			&\leq  {T} \max_{[s,S]}F+{\eps}\sup_{i\notin \Sigma^0}v_i^0 \leq C.
		\end{split}
		\]
		So all the listed uniform bounds are proved. The weak$^*$ convergences in \eqref{eq:lim1} and \eqref{eq:lim2} now directly follow from classical weak$^*$ compactness.
	\end{proof}

	We now compute the remainder terms which appear when plugging the previously introduced discrete interpolants in \eqref{eq:contlimit} (actually in \eqref{eq:weakcontlimit}). We will then prove that such remainders vanish as $N\to \infty$.
	
	To this aim, we will make use several times of the following identity
	\begin{equation}\label{eq:rhodot}
		\dot{\rho}_i(t)=-\frac1N\frac{\dot{x}_{i+1}(t)-\dot{x}_i(t)}{(x_{i+1}(t)-x_i(t))^2}=-\rho_i(t)\frac{\dot{x}_{i+1}(t)-\dot{x}_i(t)}{x_{i+1}(t)-x_i(t)},\quad \text{for $t\in(0,T)$ and $i=0,\ldots,N-1$}.
	\end{equation}
	We also recall the Leibniz rule, namely for $\varphi\in C^1([0,T]\times\R)$ there holds
	\begin{equation}\label{eq:Leibniz}
		\frac{d}{dt}\left(\int_{x_i(t)}^{x_{i+1}(t)}\varphi(t,x)\,dx\right)=\int_{x_i(t)}^{x_{i+1}(t)}\partial_t{\varphi}(t,x)\,dx+\dot{x}_{i+1}(t)\varphi(t,x_{i+1}(t))-\dot{x}_i(t)\varphi(t,x_i(t)).
	\end{equation}
	
	Finally, let us state the following simple summation by parts lemma, which will be used later on.
	\begin{lemma}\label{le:sumparts}
		For any real sequences $\{a_i\}, \{b_i\}$, and for all $N\ge 1$ one has:\[
		\sum_{i=0}^{N-1}(a_{i+1}-a_i)b_i=a_Nb_{N-1}-a_0b_0-\sum_{i=0}^{N-2}(b_{i+1}-b_i)a_{i+1}.
		\]
	\end{lemma}
	
	\begin{proof}
		It is enough to compute\[
		\begin{split}
			\sum_{i=0}^{N-1}(a_{i+1}-a_i)b_i=\sum_{i=0}^{N-1}a_{i+1}b_i-\sum_{i=0}^{N-1}a_ib_i=a_Nb_{N-1}+\sum_{i=0}^{N-2}a_{i+1}b_i-a_0b_0-\sum_{i=0}^{N-2}a_{i+1}b_{i+1},
		\end{split}
		\]
		and we conclude.
	\end{proof}
	We start our analysis by showing that the pair $(\rho^N,e_1^N)$ is an almost solution of the transport equation, i.e. the first equation in \eqref{eq:weakcontlimit}.
	\begin{proposition}\label{prop:2}
		For all $\varphi\in C^2([0,T]\times \R)$ with $\varphi(T)\equiv 0$, it holds 
		\begin{equation}\label{eq:ep2}
			\int_0^T\int_{\R}\Big(\rho^N\partial_t{\varphi}+e^N_1\partial_x\varphi\Big)\,dxdt+\int_{\R}\rho^N(0)\varphi(0)\,dx=\frac{\mathcal R^N(\varphi)}{N},
		\end{equation}
		where the term $\mathcal R^N(\varphi)$ is given by \[
		\mathcal R^N(\varphi)=\frac12\sum_{i=0}^{N-1}\int_0^T\Big(\dot{x}_{i+1}(t)-\dot{x}_i(t)\Big)\left(\int_{x_i(t)}^{x_{i+1}(t)}\partial_{xx}\varphi(t,x)\displaystyle\left(\frac{x-x_i(t)}{x_{i+1}(t)-x_i(t))}\right)^2dx-\partial_x\varphi(t,x_{i+1}(t))\right)\,dt.
		\]
	\end{proposition}
	
	\begin{proof}
		Using the definitions of $\rho^N$ and $e_1^N$ we obtain\[
		\begin{split}
			&\int_0^T\int_{\R}\Big(\rho^N\partial_t{\varphi}+e_1^N\partial_x\varphi\Big)\,dxdt\\
			&=\sum_{i=0}^{N-1}\int_0^T\rho_i(t)\left(\int_{x_i(t)}^{x_{i+1}(t)}\partial_t{\varphi}(t,x)\,dx+\dot{x}_i(t)\Big(\varphi(t,x_{i+1}(t))-\varphi(t,x_{i}(t))\Big)\right)\,dt=(*).
		\end{split}
		\]
		By exploiting \eqref{eq:Leibniz} and integrating by parts in time we continue the above equality
		\[
		\begin{split}
			(*)=&-\sum_{i=0}^{N-1}\rho_i^0\int_{x_i^0}^{x_{i+1}^0}\varphi(0,x)\,dx\\
			&-\sum_{i=0}^{N-1}\int_0^T\left(\dot{\rho}_i(t)\int_{x_i(t)}^{x_{i+1}(t)}\varphi(t,x)\,dx+\rho_i(t)\varphi(t,x_{i+1}(t))\Big(\dot{x}_{i+1}(t)-\dot{x}_i(t)\Big)\right)\,dt.
		\end{split}
		\]
		By using \eqref{eq:rhodot}, we now infer
		\[
		\begin{split}
			(*)=&-\int_{\R}\rho^N(0)\varphi(0,x)\,dx\\
			&+\frac1N\sum_{i=0}^{N-1}\displaystyle\int_0^T(\dot{x}_{i+1}(t)-\dot{x}_i(t))\displaystyle\frac{\displaystyle\int_{x_i(t)}^{x_{i+1}(t)}\varphi(t,x)\,dx-\Big({x}_{i+1}(t)-{x}_i(t)\Big)\varphi(t,x_{i+1}(t))}{(x_{i+1}(t)-x_i(t))^2}\,dt.
		\end{split}
		\]
		We can then conclude by using Taylor expansion with integral remainder:\[
		\begin{split}
			&\displaystyle\frac{\displaystyle\int_{x_i(t)}^{x_{i+1}(t)}\varphi(t,x)\,dx-\Big({x}_{i+1}(t)-{x}_i(t)\Big)\varphi(t,x_{i+1}(t))}{(x_{i+1}(t)-x_i(t))^2}\\
			&=\frac12\left(\int_{x_i(t)}^{x_{i+1}(t)}\partial_{xx}\varphi(t,x)\displaystyle\frac{(x-x_i(t))^2}{(x_{i+1}(t)-x_i(t))^2}\,dx-\partial_x\varphi(t,x_{i+1}(t))\right).
		\end{split}
		\]
	\end{proof}
	Let us now show that the numerator $\mathcal R^N(\varphi)$ of the remainder term in \eqref{eq:ep2} is uniformly bounded.
	\begin{proposition}\label{prop:3}
		With the notations and the assumptions of Proposition~\ref{prop:2}, we  have that 
		\begin{equation}\label{eq:boundR}
			|\mathcal R^N(\varphi)|\le  C \Big(  \| \partial_x \varphi\|_{L^\infty((0,T)\times(s,S))} +  \| \partial_{xx} \varphi\|_{L^\infty((0,T)\times(s,S))} \Big), 
		\end{equation}
		where the constant $C$ depends neither on $N$ nor on $\eps$.
	\end{proposition}
	
	\begin{proof}
		By the explicit expression of $\mathcal R^N(\varphi)$ we can estimate\[
		\begin{split}
			|\mathcal R^N(\varphi)|&\leq \frac12\sum_{i=0}^{N-1}\int_0^T\Big(\dot{x}_{i+1}(t)+\dot{x}_i(t)\Big)\int_{x_i(t)}^{x_{i+1}(t)}|\partial_{xx}\varphi(t,x)|\left(\frac{x-x_i(t)}{x_{i+1}(t)-x_i(t)}\right)^2\,dx\,dt\\
			&\hspace{1cm}+\frac12\left|\sum_{i=0}^{N-1}\int_0^T\Big(\dot{x}_{i+1}(t)-\dot{x}_i(t)\Big)\partial_x\varphi(t,x_{i+1}(t))\,dt\right|=:I^N+II^N.
		\end{split}
		\]
		For the sake of clarity, in the following estimates we will use the shortcut $\|\cdot\|_\infty$ in place of the heavy expression $\|\cdot\|_{L^\infty((0,T)\times(s,S))}$. The first term $I^N$ can be bounded by
		\begin{align}
			I^N\le &\frac{\|\partial_{xx}\varphi\|_{\infty}}{2}\sum_{i=0}^{N-1}\int_0^T (\dot{x}_{i+1}(t)+\dot{x}_i(t))({x}_{i+1}(t)-{x}_i(t))\,dt\nonumber\\
			=& \frac{\|\partial_{xx}\varphi\|_{\infty}}{2}\left(\int_0^T (\dot{x}_{N}(t){+}\dot{x}_{N-1}(t))({x}_{N}(t){-}{x}_{N-1}(t))\,dt+\sum_{i=0}^{N-2}\int_0^T (\dot{x}_{i+1}(t){+}\dot{x}_i(t))({x}_{i+1}(t){-}{x}_i(t))\,dt\right)\nonumber\\
			\le &\frac{\|\partial_{xx}\varphi\|_{\infty}}{2}\left[(S{-}s)(x_N(T){-}x_N^0{+}x_{N-1}(T){-}x_{N-1}^0){+}2\sup_{i=0,\ldots,N-1}\|\dot x_i\|_{L^\infty(0,T)}\int_0^T(x_{N-1}(t){-}x_0(t))\, dt\right]\nonumber\\
			\le &\|\partial_{xx}\varphi\|_{\infty}(S-s)\left(S-s+T\sup_{i=0,\ldots,N-1}\|\dot x_i\|_{L^\infty(0,T)}\right).\label{eq:estimate1}
		\end{align}
		Concerning the second term $II^N$ we exploit Lemma~\ref{le:sumparts} obtaining
		\begin{align}\label{eq:estimate2}
			II^N=&\frac 12 \left|\int_0^T\!\!(\dot x_N(t)\partial_x\varphi(t,x_N(t)){-}\dot x_0(t)\partial_x\varphi(t,x_1(t))\,dt{-}\!\sum_{i=0}^{N-2}\!\int_0^T\!\!(\partial_x\varphi(t,x_{i+2}(t){-}\partial_x\varphi(t,x_{i+1}(t)) \dot x_{i+1}(t)\, dt\right|\nonumber\\
			\le &\frac{\|\partial_{x}\varphi\|_{\infty}}{2}(x_N(T)-x_N^0+x_0(T)-x_0^0)+\frac{\|\partial_{xx}\varphi\|_{\infty}}{2}\int_0^T\sum_{i=0}^{N-2}\dot x_{i+1}(t)(x_{i+2}(t)-x_{i+1}(t))\,dt\nonumber\\
			\le &\|\partial_{x}\varphi\|_{\infty}(S-s)+\frac{\|\partial_{xx}\varphi\|_{\infty}}{2}\sup_{i=0,\ldots,N-1}\|\dot x_i\|_{L^\infty(0,T)}\int_0^T(x_N(t)-x_1(t))\,dt\nonumber\\
			\le& (S-s)\left(\|\partial_{x}\varphi\|_{\infty}+T\frac{\|\partial_{xx}\varphi\|_{\infty}}{2}\sup_{i=0,\ldots,N-1}\|\dot x_i\|_{L^\infty(0,T)} \right).
		\end{align}
		By combining \eqref{eq:estimate1} and \eqref{eq:estimate2} we now conclude by means of \eqref{eq:uniflipbound} and \eqref{eq:initialassumption}.
	\end{proof}
	We now focus on the second equation in \eqref{eq:weakcontlimit}.
	\begin{proposition}\label{prop:4}
		For all $\varphi\in C^1([0,T]\times \R)$ with $\varphi(T)\equiv 0$, it holds
		\begin{equation}\label{eq:ep3}
			\int_0^T\int_{\R}\left[\eps(e_1^N\partial_t{\varphi}+e_2^N\partial_x\varphi)+\rho^NF\varphi-\gamma\frac{e_1^N}{\vartheta(\rho^N)}\varphi\right]\,dxdt+\eps\int_{\R}e^N_1(0)\varphi(0)\,dx=\frac{\mathcal S^N(\varphi)}{N},   
		\end{equation}
		where the term $\mathcal S^N(\varphi)$ is given by \[
		\begin{split}
			\mathcal S^N(\varphi)=N\sum_{i=0}^{N-1}\int_0^T\rho_i(t)\int_{x_i(t)}^{x_{i+1}(t)}\Big(F(t,x)-F(t,x_i(t))\Big)\varphi(t,x)\,dxdt.
		\end{split}
		\]
	\end{proposition}
	
	\begin{proof}
		By using the definitions of $e^N_1$ and $e^N_2$ we compute 
		\[
		\begin{split}
			&\int_0^T\int_{\R}\left[\eps(e_1^N\partial_t{\varphi}+e_2^N\partial_x\varphi)+\rho^NF\varphi-\gamma\frac{e_1^N}{\vartheta(\rho^N)}\varphi\right]\,dxdt\\
			&=\sum_{i=0}^{N-1}\int_0^T\Big[\eps\rho_i(t)\dot{x}_i(t)\int_{x_i(t)}^{x_{i+1}(t)}\partial_t{\varphi}(t,x)\,dx\\
			&\hspace{3cm}+\eps \rho_i(t)\dot{x}_i(t)\int_{x_i(t)}^{x_{i+1(t)}} \left(\dot{x}_i(t)+\frac{\dot{x}_{i+1}(t)-\dot{x}_{i}(t)}{x_{i+1}(t)-x_i(t)}(x-x_i(t))\right)\partial_x\varphi(t,x)\,dx\\
			&\hspace{3cm}+\rho_i(t)\int_{x_i(t)}^{x_{i+1}(t)}F(t,x)\varphi(t,x)\,dx-\gamma\rho_i(t)\frac{\dot{x}_i(t)}{\vartheta(\rho_i(t))}\int_{x_i(t)}^{x_{i+1}(t)}\varphi(t,x)\,dx\Big]\,dt=(*).
		\end{split}
		\]
		Exploiting \eqref{eq:rhodot}, \eqref{eq:Leibniz}, and integrating by parts in time the first term of the right-hand side and in space the second one, recalling also that $\varphi(T)\equiv 0$, we can continue the above chain of equalities obtaining
		\[
		\begin{split}
			&(*)=-\eps\sum_{i=0}^{N-1}\rho_i^0\dot{x}_i(0)\int_{x_i^0}^{x_{i+1}^0}\varphi(0,x)\,dx\\
			&\hspace{1cm}-\eps\sum_{i=0}^{N-1}\int_0^T\Big[\Big(\dot{\rho}_i(t)\dot{x}_i(t)+\rho_i(t)\ddot{x}_i(t)\Big)\int_{x_i(t)}^{x_{i+1}(t)}\varphi(t,x)\,dx\\
			&\hspace{3cm}+\rho_i(t)\dot{x}_i(t)\Big(\dot{x}_{i+1}(t)\varphi(t,x_{i+1}(t))-\dot{x}_{i}(t)\varphi(t,x_i(t))\Big)\Big]\,dt\\
			&\hspace{1cm}+\eps\sum_{i=0}^{N-1}\int_0^T\rho_i(t)\dot{x}_i(t)\left(\dot{x}_{i+1}(t)\varphi(t,x_{i+1}(t)){-}\dot{x}_{i}(t)\varphi(t,x_i(t)){-}\frac{\dot{x}_{i+1}(t){-}\dot{x}_i(t)}{x_{i+1}(t){-}x_i(t)}\int_{x_i(t)}^{x_{i+1}(t)}\!\!\!\!\!\varphi(t,x)\,dx\right)\,dt\\
			&\hspace{1cm}+\sum_{i=0}^{N-1}\int_0^T\Big[\rho_i(t)\int_{x_i(t)}^{x_{i+1}(t)}F(t,x)\varphi(t,x)\,dx-\gamma\rho_i(t)\frac{\dot{x}_i(t)}{\vartheta(\rho_i(t))}\int_{x_i(t)}^{x_{i+1}(t)}\varphi(t,x)\,dx\Big]\,dt
		\end{split}
		\]
		\[
		\begin{split}
			&\hspace{-3cm}=-\eps\int_{\R}e_1^N(0,x)\varphi(0,x)\,dx-\sum_{i=0}^{N-1}\int_0^T\rho_i(t)\Big(\eps\ddot{x}_i(t)+\gamma\frac{\dot{x}_i(t)}{\vartheta(\rho_i(t))}\Big)\int_{x_i(t)}^{x_{i+1}(t)}\varphi(t,x)\,dx\\
			&\hspace{1cm}+\sum_{i=0}^{N-1}\int_0^T\rho_i(t)\int_{x_i(t)}^{x_{i+1}(t)}F(t,x)\varphi(t,x)\,dxdt.\\
		\end{split}
		\]
		By using equation \eqref{eq:systempos}, with $\zeta=\vartheta$, and recalling that we may divide by $\vartheta(\rho_i(t))$ since it is positive in $(0,T]$ by Proposition~\ref{prop:nosaturated} and Remark~\ref{rmk:nosaturated} we finally infer
		\[
		\begin{split}
			&\hspace{-3cm}(*)=-\eps\int_{\R}e_1^N(0,x)\varphi(0,x)\,dx+\sum_{i=0}^{N-1}\int_0^T\rho_i(t)\int_{x_i(t)}^{x_{i+1}(t)}\Big(F(t,x)-F(t,x_i(t))\Big)\varphi(t,x)\,dxdt,\\
		\end{split}
		\]
		and we conclude.
	\end{proof}
	We finally prove that also the term $\mathcal S^N(\varphi)$ is uniformly bounded.
	\begin{proposition}\label{prop:5}
		With the notations and the assumptions of Proposition~\ref{prop:4}, we have that
		\begin{equation}\label{eq:boundS}
			|\mathcal S^N(\varphi)|\le L_FT(S-s)\|\varphi\|_{L^\infty((0,T)\times(s,S))},
		\end{equation}
		where $L_F$ denotes the supremum in $[0,T]$ of the Lipschitz constants of the map $x\mapsto F(t,x)$.
	\end{proposition}
	
	\begin{proof}
		It is enough to observe that
		\[
		\begin{split}
			{|\mathcal S^N(\varphi)|}&\leq \sum_{i=0}^{N-1}\int_0^T\int_{x_i(t)}^{x_{i+1}(t)}\frac{|F(t,x)-F(t,x_{i}(t))|}{x_{i+1}(t)-x_i(t)}|\varphi(t,x)|\,dxdt\\
			&\leq L_F\|\varphi\|_{L^\infty((0,T)\times(s,S))}\sum_{i=0}^{N-1}\int_0^T(x_{i+1}(t)-x_i(t))\,dt\\
			&=L_F\|\varphi\|_{L^\infty((0,T)\times(s,S))}\int_0^T(x_N(t)-x_0(t))\,dt\leq L_F\|\varphi\|_{L^\infty((0,T)\times\R)}T(S-s).
		\end{split}
		\]
	\end{proof}
	We are now in position to prove Theorem~\ref{thm:maincont}.
	\begin{proof}[Proof of Theorem~\ref{thm:maincont}]
		The uniform bounds on the initial data and on $x_N$ \eqref{eq:initialassumption} together with the uniform estimate in \eqref{eq:uniflipbound} directly yield \eqref{eq:limitsa}, by means of Ascoli-Arzel\`a Theorem. The convergences in \eqref{eq:limitsc} and \eqref{eq:limitsd} have already been proved in Proposition~\ref{prop:uniformbounds}, and similarly one can show \eqref{eq:limitsb}. The statements on the supports then follow by weak lower semicontinuity, since by definition the supports of $\rho^N, e_1^N$ and $e_2^N$ are contained in
		\[
		\bigcup_{t\in[0,T]}\{t\}\times [x_0(t),x_N(t)].
		\]
		
		We are just let to prove that the limit quadruple $(\rho,e_1,e_2,\lambda)$ solves~\eqref{eq:contlimit} in the sense of Definition~\ref{def:solcont}. In particolar, we only need to check condition~\ref{ciii}, since conditions~\ref{ci} and~\ref{cii} are automatically fulfilled because the interpolants \eqref{eq:definterpolants} are nonnegative functions (with uniformly bounded supports).
		To do so, we let $N\to \infty$ in \eqref{eq:ep2} and \eqref{eq:ep3}. By exploiting Propositions~\ref{prop:3} and \ref{prop:5}, together with the known weak convergence of the involved quantities, we deduce that the limit functions solve system \eqref{eq:weakcontlimit} for all $\varphi\in C^2([0,T]\times \R)$ with $\varphi(T)\equiv 0$. A simple density argument now shows that the equation is satisfied also for $\varphi\in C^1([0,T]\times\R)$ with $\varphi(T)\equiv 0$, and so we conclude.
	\end{proof}

	\section{First order approximation} \label{sec:vaninertia}
	
	In this last section we prove Theorem~\ref{thm:vaninert}. We recall that here the parameter $\eps=\eps_N$ vanishes as $N\to \infty$. The strategy of the proof follows the lines of previous section, with the only difference that in the current situation we will show in addition that all terms multiplied by $\eps_N$ actually go to zero in the many-particle limit.
	
	\begin{proof}[Proof of Theorem~\ref{thm:vaninert}]
		Thanks to the assumption \eqref{eq:initialassumption} and by the uniform estimate \eqref{eq:uniflipbound} we deduce that $x_0$ and $x_N$ are bounded in $W^{1,\infty}(0,T)$, hence by Ascoli-Arzel\`a Theorem we infer \eqref{eq:limits4}. In Proposition~\ref{prop:uniformbounds} we already proved the validity of \eqref{eq:limits4c} and \eqref{eq:limits4d}, and similarly also \eqref{eq:limits4b} can be deduced. Moreover, the statements on the sign and on the supports can easily be checked by arguing as in the proof of Theorem~\ref{thm:maincont}.
		
		Proposition~\ref{prop:uniformbounds} also ensures that the sequence $\frac{e_1^N}{\vartheta(\rho^N)}\,dtdx$ weakly$^*$ converges in the sense of measures to some (positive) Radon measure $\lambda$ on $[0,T]\times \R$. In order to conclude, we thus need to show that $\lambda = \frac 1\gamma\rho F\,dtdx$ and prove that $\rho$ and $e_1$ satisfy ~\eqref{eq:CE}. 
		
		To this aim, we first fix $\varphi \in C^2([0,T]\times \R)$ with $\varphi(T) \equiv 0$ and we recall that Proposition~\ref{prop:2} established that 
		\[ \int_0^T\int_{\R}\Big(\rho^N\partial_t{\varphi}+e^N_1\partial_x\varphi\Big)\,dxdt+\int_{\R}\rho^N(0)\varphi(0)\,dx=\frac{\mathcal R^N(\varphi)}{N}.
		\]
		Now, by \eqref{eq:boundR} we deduce that 
		\[  \frac{\mathcal |R^N(\varphi)|}{N} \leq C \frac{\| \partial_x \varphi\|_{L^\infty((0,T)\times(s,S))} + \| \partial_{xx} \varphi\|_{L^\infty((0,T)\times(s,S))}}{N} \, \longrightarrow\, 0,\qquad \text{as }N\to \infty, \]
		thus \eqref{eq:CE} is proved (by density one deduces its validity for $\varphi \in C^1([0,T]\times \R)$).
		
		On the other hand, arguing similarly as in Propositions~\ref{prop:4}, for any $\varphi \in C^1([0,T]\times \R)$ (without the request on the final time condition) we obtain
		\begin{equation}\label{eq:intestep vanishing inertia}
			\begin{aligned}
				\eps_N\int_0^T\int_{\R} (e_1^N\partial_t{\varphi}+e_2^N\partial_x\varphi)\,dxdt &+ \int_0^T\int_{\R} (\rho^NF-\gamma\frac{e_1^N}{\vartheta(\rho^N)})\varphi\,dxdt\\
				&+\eps_N\int_{\R} \big(e^N_1(0)\varphi(0) - e^N_1(T)\varphi(T)\big)\,dx =\frac{\mathcal S^N(\varphi)}{N},
			\end{aligned}
		\end{equation}
		where the remainder $\mathcal S^N(\varphi)/N $ vanishes as $N\to \infty$ due to \eqref{eq:boundS}. By using \eqref{eq:uniflipbound}, the uniform bounds of Proposition~\ref{prop:uniformbounds}, and recalling \eqref{eq:initialassumption}, it is then easy to see that all the term multiplied by $\eps_N$ in \eqref{eq:intestep vanishing inertia} are bounded by a universal constant. Since $\eps_N\to 0$, from the weak$^*$-type convergences $\rho^N\overset{*}{\rightharpoonup}  \rho$ and $\frac{e_1^N}{\vartheta(\rho^N)}\,dtdx\overset{*}{\rightharpoonup} \lambda$, we deduce that for every $\varphi \in C^1([0,T] \times \R)$ the following identity holds true:
		\[ \int_0^T \int_\R \rho F \varphi \, dx dt = \gamma\int_{[0,T] \times \R} \varphi\, d\lambda.  \] 
		Since, by density, the above equality is still true for all $\varphi \in C^0 ([0,T] \times\R)$ we finally infer that $\lambda = \frac 1\gamma\rho F\,dtdx$ and we conclude the proof.
	\end{proof}

	\bigskip
	
	\noindent\textbf{Acknowledgements.}
	We warmly thank Marco Di Francesco for some useful exchanges on the topic of the paper.  
	E. Radici acknowledges the organisers of the Workshop \lq\lq Particle Systems in Dynamics, Optimization, and Learning\rq\rq which took place at the Lagrange Mathematics and Computation Research Center of Paris, where some improvements of an earlier version of this work have been suggested. F. Riva acknowledges the Department of Mathematics ``F. Casorati'' of the University of Pavia, to which he was affiliated while most of the research of this work was carried out. E. Radici and F. Riva have been partially supported by the INdAM-GNAMPA project 2023  CUP\_E53C22001930001.
	D. Mazzoleni has been partially supported by the MUR via PRIN projects P2022R537CS and P2020F3NCPX, and by the INdAM-GNAMPA project 2024 CUP\_E53C23001670001.

	\bigskip

	\bibliography{bibliotraffic2}{}
	\bibliographystyle{siam}

	\medskip
	\small
	\begin{flushleft}
		\noindent \verb"dario.mazzoleni@unipv.it"\\
		Dipartimento di Matematica  ``F. Casorati'', \\
		Universit\`a di Pavia,\\
		via Ferrata 5, 27100 Pavia, Italy\\
		\smallskip
		\noindent \verb"emanuela.radici@univaq.it"\\
		DISIM - Dipartimento di Ingegneria e Scienze dell'Informazione e Matematica, \\
		Universit\`a de L'Aquila, \\
		via Vetoio 1 (Coppito), 67100 L'Aquila, Italy\\
		\smallskip
		\noindent \verb"filippo.riva@unibocconi.it"\\
		Department of Decision Sciences, \\
		Bocconi University,\\
		via Roentgen 1, 20136 Milano, Italy\\
	\end{flushleft}
	
\end{document}